\documentclass[oneside,english,british]{amsart}
\usepackage{lmodern}

\usepackage[T1]{fontenc}
\usepackage[latin9]{inputenc}
\usepackage{geometry}
\usepackage{color}
\usepackage{mathrsfs}
\usepackage{amsthm}
\usepackage{amstext}
\usepackage{amssymb}
\usepackage{esint}

\makeatletter
\numberwithin{equation}{section}
\numberwithin{figure}{section}
\theoremstyle{plain}
\newtheorem{thm}{\protect\theoremname}[section]
  \theoremstyle{remark}
  \newtheorem{rem}[thm]{\protect\remarkname}
  \theoremstyle{plain}
  \newtheorem{lem}[thm]{\protect\lemmaname}
  \theoremstyle{plain}
  \newtheorem{prop}[thm]{\protect\propositionname}
  \theoremstyle{plain}
  \newtheorem{cor}[thm]{\protect\corollaryname}

\usepackage{bbold}
\usepackage{tikz}
\usepackage{caption}
\usepackage{subcaption}

\makeatother

\usepackage{babel}
  \addto\captionsbritish{\renewcommand{\corollaryname}{Corollary}}
  \addto\captionsbritish{\renewcommand{\lemmaname}{Lemma}}
  \addto\captionsbritish{\renewcommand{\propositionname}{Proposition}}
  \addto\captionsbritish{\renewcommand{\remarkname}{Remark}}
  \addto\captionsbritish{\renewcommand{\theoremname}{Theorem}}
  \addto\captionsenglish{\renewcommand{\corollaryname}{Corollary}}
  \addto\captionsenglish{\renewcommand{\lemmaname}{Lemma}}
  \addto\captionsenglish{\renewcommand{\propositionname}{Proposition}}
  \addto\captionsenglish{\renewcommand{\remarkname}{Remark}}
  \addto\captionsenglish{\renewcommand{\theoremname}{Theorem}}
  \providecommand{\corollaryname}{Corollary}
  \providecommand{\lemmaname}{Lemma}
  \providecommand{\propositionname}{Proposition}
  \providecommand{\remarkname}{Remark}
\providecommand{\theoremname}{Theorem}

\begin{document}

\title{\textup{The $n$-point correlation of quadratic forms}.}

\author{Oliver Sargent}

\address{\textsc{Department Of Mathematics, University Walk, Bristol, BS8
1TW, UK.}\texttt{ }}

\email{\texttt{Oliver.Sargent@bris.ac.uk}}
\begin{abstract}
In this paper we investigate the distribution of the set of values
of a quadratic form $Q$, at integral points. In particular we are
interested in the $n$-point correlations of the this set. The asymptotic
behaviour of the counting function that counts the number of $n$-tuples
of integral points $\left(v_{1},\dots,v_{n}\right)$, with bounded
norm, such that the $n-1$ differences $Q\left(v_{1}\right)-Q\left(v_{2}\right),\dots Q\left(v_{n-1}\right)-Q\left(v_{n}\right)$,
lie in prescribed intervals is obtained. The results are valid provided
that the quadratic form has rank at least 5, is not a multiple of
a rational form and $n$ is at most the rank of the quadratic form.
For certain quadratic forms satisfying Diophantine conditions we obtain
a rate for the limit. The proofs are based on those in the recent
preprint (\cite{G-M}) of F. G\"otze and G. Margulis, in which they
prove an `effective' version of the Oppenheim Conjecture. In particular,
the proofs rely on Fourier analysis and estimates for certain theta
series. 
\end{abstract}
\maketitle

\section{Introduction.}

\subsection{Background}

Let $Q:\mathbb{R}^{d}\rightarrow\mathbb{R}$ be a quadratic form.
It is interesting to understand the distribution of the set $Q\left(\mathbb{Z}^{d}\right)$
inside $\mathbb{R}$. If there exists a multiple of $Q$ such that
its coefficients are all rational, then $Q$ is called a rational
form. In the case when $Q$ is rational, $Q\left(\mathbb{Z}^{d}\right)$
is a discrete set inside $\mathbb{R}$. When $Q$ is not a rational
form, $Q$ is called an irrational form. For irrational forms, the
first milestone in understanding the distribution of the set $Q\left(\mathbb{Z}^{d}\right)$
inside $\mathbb{R}$ was reached by G. Margulis in \cite{MR993328}
when he provided a proof (shortly afterwards, refined by S.G. Dani
and G. Margulis in \cite{MR1016271}) of the `Oppenheim Conjecture'.
The modern statement of which is as follows: if $d\geq3$ and $Q$
is a nondegenerate, irrational and indefinite form, then $Q\left(\mathbb{Z}^{d}\right)$
is dense in $\mathbb{R}$. 

Once it is known that $Q\left(\mathbb{Z}^{d}\right)$ is dense in
$\mathbb{R}$, one can ask for a more precise answer to the question
of how $Q\left(\mathbb{Z}^{d}\right)$ is distributed in $\mathbb{R}$.
Let $I\subset\mathbb{R}$ be any interval and $E_{Q}\left(I,T\right)=\left\{ v\in\mathbb{R}^{d}:Q\left(v\right)\in I,\left\Vert v\right\Vert \leq T\right\} $,
then one can ask for an asymptotic formula for the size of the set
$\mathbb{Z}^{d}\cap E_{Q}\left(I,T\right)$. The first results in
this direction were obtained by Dani-Margulis in \cite{MR1237827}
who proved, if $d\geq3$ and $Q$ is a nondegenerate, irrational and
indefinite form and $I$ is any interval, then 
\[
\lim_{T\rightarrow\infty}\frac{\left|\mathbb{Z}^{d}\cap E_{Q}\left(I,T\right)\right|}{\mathrm{Vol}\left(E_{Q}\left(I,T\right)\right)}\geq1.
\]
The situation regarding the upper bounds is more delicate and this
was dealt with by the work of A. Eskin, G. Margulis and S. Mozes in
\cite{MR1609447} who proved that if $d\geq5$ and $Q$ is a nondegenerate,
irrational and indefinite form and $I$ is any interval, then 
\begin{equation}
\lim_{T\rightarrow\infty}\frac{\left|\mathbb{Z}^{d}\cap E_{Q}\left(I,T\right)\right|}{\mathrm{Vol}\left(E_{Q}\left(I,T\right)\right)}=1.\label{eq:into 2}
\end{equation}
It should be noted that the actual results from \cite{MR1237827}
and \cite{MR1609447} are more general than stated above. The situation
regarding the upper bounds for the case when $d=3$ or $4$ is particularly
interesting and is also considered in \cite{MR1609447}. In the cases
when $Q$ has signature $\left(2,1\right)$ or $\left(2,2\right)$
no asymptotic formula of the form (\ref{eq:into 2}) is possible for
general quadratic forms, since in these cases there exist examples
of quadratic forms for which (\ref{eq:into 2}) fails. In \cite{MR2153398},
quadratic forms of signature $\left(2,2\right)$ satisfying a slightly
modified version of (\ref{eq:into 2}) are characterised by certain
Diophantine conditions. The work of Eskin-Margulis-Mozes can be interpreted
as providing conditions which ensure the set $Q\left(\mathbb{Z}^{d}\right)$
is equidistributed in $\mathbb{R}$. 

One can ask still finer questions about the distribution of the set
$Q\left(\mathbb{Z}^{d}\right)$. Let $e_{1},\dots,e_{nd}$ be the
standard basis of $\mathbb{R}^{nd}$, let $p_{1}\dots,p_{n}$ denote
the projections onto $\left\langle e_{1},\dots,e_{d}\right\rangle ,\dots,\left\langle e_{\left(n-1\right)d+1},\dots,e_{nd}\right\rangle $
respectively. For $v\in\mathbb{R}^{nd}$ and $1\leq i\leq n$, we
will write $v_{i}=p_{i}\left(v\right)$. Let $I_{1},\dots,I_{n-1}$
be intervals and 
\[
P_{Q}^{n}\left(I_{1},\dots,I_{n-1},T\right)=\left\{ v\in\mathbb{R}^{nd}:Q\left(v_{1}\right)-Q\left(v_{2}\right)\in I_{1},\dots,Q\left(v_{n-1}\right)-Q\left(v_{n}\right)\in I_{n-1},\left\Vert v\right\Vert \leq T\right\} .
\]
In order to understand the $n$-point correlations of the set $Q\left(\mathbb{Z}^{d}\right)$,
one asks for an asymptotic formula for the size of the set $\mathbb{Z}^{nd}\cap P_{Q}^{n}\left(I_{1},\dots,I_{n-1},T\right)$.
A more general problem about the distribution of values at integral
points of systems of quadratic forms was studied by W. M\"uller in
\cite{MR2379669}. In particular, it follows from Theorem 1 of \cite{MR2379669}
that if $n\geq2$, $4n\leq d$ and $Q$ is a nondegenerate and irrational
form, then 
\begin{equation}
\lim_{T\rightarrow\infty}\frac{\left|\mathbb{Z}^{nd}\cap P_{Q}^{n}\left(I_{1},\dots,I_{n-1},T\right)\right|}{\mathrm{Vol}\left(P_{Q}^{n}\left(I_{1},\dots,I_{n-1},T\right)\right)}=1.\label{eq:intro1}
\end{equation}
When $n=2$ and $d\geq3$, it is easy to see that (\ref{eq:intro1})
follows from the main Theorem of \cite{MR1609447}. For positive definite
forms the $n$-point correlation problem was also studied by M\"uller.
In \cite{MR2747345}, M\"uller obtains the following result: if $d\geq4$
and $Q$ is a nondegenerate, irrational and positive definite form,
then (\ref{eq:intro1}) holds for every $n$. In \cite{MR2747345}
M\"uller formulates the problem in slightly different language, but
it is easily seen to be equivalent to the form stated here up to a
change of variables and modifications of the norms involved. The main
result of this paper extends the results of M\"uller to a larger
range of $n$ for indefinite forms.

\subsection{Statement of results.}

Using the notation from the previous subsection we can now state the
main results. 
\begin{thm}
\label{thm:Pair correlation}Suppose that $Q$ is not a multiple of
a rational form and $d\geq5$ and $2<n\leq d$. Then, for any intervals
$I_{1},\dots,I_{n-1}$,
\[
\lim_{T\rightarrow\infty}\frac{\left|\mathbb{Z}^{nd}\cap P_{Q}^{n}\left(I_{1},\dots,I_{n-1},T\right)\right|}{\mathrm{Vol}\left(P_{Q}^{n}\left(I_{1},\dots,I_{n-1},T\right)\right)}=1.
\]
Moreover, there exists a positive constant $C_{Q,n}$, depending only
on $Q$ and $n$, such that for any intervals $I_{1},\dots,I_{n-1}$,
\[
\lim_{T\rightarrow\infty}\frac{\left|\mathbb{Z}^{nd}\cap P_{Q}^{n}\left(I_{1},\dots,I_{n-1},T\right)\right|}{T^{nd-2\left(n-1\right)}}=C_{Q,n}\prod_{i=1}^{n-1}\left|I_{i}\right|.
\]

\end{thm}
For quadratic forms $Q$ satisfying the following Diophantine condition
it is possible to prove an effective version of Theorem \ref{thm:Pair correlation}.
Let $Q$ also denote the symmetric $d\times d$ matrix that is associated
to the quadratic form $Q$. Let $0<\kappa<1$ and $A>0$, say that
$Q$ is of type $\left(\kappa,A\right)$ if for every $M\in\mathrm{Mat}_{d}\left(\mathbb{Z}\right)$
and $q\in\mathbb{Z}\setminus\left\{ 0\right\} $ we have 
\[
\inf_{t\in\left[1,2\right]}\left\Vert Mq^{-1}-tQ\right\Vert \geq Aq^{-1-\kappa}.
\]
The size of $\kappa$ depends on how well $Q$ can be approximated
by a rational matrix, if $\kappa$ is close to $1$, then $Q$ is
in some sense close to a rational matrix. 
\begin{thm}
\label{thm:Pair correlation-1}Suppose that $Q$ is of Diophantine
type $\left(\kappa,A\right)$ and $d\geq5$ and $2<n\leq d$. Let
$\delta\left(\kappa\right)=\frac{2\left(d-4\right)\left(1-\kappa\right)}{\left(1+nd\right)\left(d+1+\kappa\right)}.$
Then, for any intervals $I_{1},\dots,I_{n-1}$ there exists $T_{0}>0$
and a constant $C$ such that for all $T\geq T_{0}$, 
\[
\left|\frac{\left|\mathbb{Z}^{nd}\cap P_{Q}^{n}\left(I_{1},\dots,I_{n-1},T\right)\right|}{\mathrm{Vol}\left(P_{Q}^{n}\left(I_{1},\dots,I_{n-1},T\right)\right)}-1\right|\leq C\log^{n-1}\left(T\right)T^{-\delta\left(\kappa\right)}.
\]
\end{thm}
\begin{rem}
The constant $C$ appearing in Theorem \ref{thm:Pair correlation-1}
depends on $Q,n,A$ and the intervals $I_{1},\dots,I_{n-1}$. 
\end{rem}

\begin{rem}
In Theorem \ref{thm:Pair correlation-1} we use $\left\Vert .\right\Vert $
to denote the Euclidean norm. The exponent $\delta\left(\kappa\right)$
depends on the choice of norm and is possibly non optimal. If the
maximum norm was chosen, the bounds in subsection \ref{sub:Norms}
could be improved, and $\delta\left(\kappa\right)$ could be replaced
with $\delta'\left(\kappa\right)=\frac{2\left(1-\kappa\right)}{\left(d+1+\kappa\right)}$
at the cost of a factor of $\log^{nd}\left(T\right)$ appearing. 
\end{rem}

\begin{rem}
The second parts of Theorems \ref{thm:Pair correlation} and \ref{thm:Pair correlation-1}
follow easily from the first parts and the following assertion: For
any intervals $I_{1},\dots,I_{n-1}$ there exists a positive constant
$C_{Q,n}$, depending only on $Q$ and $n$, such that 
\[
\lim_{T\rightarrow\infty}\frac{1}{T^{nd-2\left(n-1\right)}}\mathrm{Vol}\left(P_{Q}^{n}\left(I_{1},\dots,I_{n-1},T\right)\right)=C_{Q,n}\prod_{i=1}^{n-1}\left|I_{i}\right|.
\]
This statement is proved in Corollary \ref{cor:volume}.
\end{rem}

\subsection{The Berry-Tabor Conjecture}

For positive definite forms there is a similar problem about the $n$-point
correlations of the normalised values of $Q$ at integral points.
This problem is discussed in \cite{MR2379669} and is interesting
because it is related to the so called Berry-Tabor Conjecture (see
\cite{Berry15091977}). A special case of this Conjecture states that
the spacings of eigenvalues of the Laplacian on `generic' multidimensional
tori should have a Poisson distribution. This problem has been studied
in \cite{MR1472786} by P. Sarnak, in \cite{MR1682221} and \cite{MR1748175}
by J. VanderKam and in \cite{MR1940408} by J. Marklof.

\subsection{Outline of paper and summary of the methods. }

One can try to prove Theorems \ref{thm:Pair correlation} and \ref{thm:Pair correlation-1}
by using the theory of unipotent flows, in analogy to what was done
in \cite{MR1609447}. The problem one encounters, is that the subgroup
of linear transformations of $\mathbb{R}^{nd}$ stabilising the quadratic
forms $Q\left(v_{1}\right)-Q\left(v_{2}\right),\dots,Q\left(v_{n-1}\right)-Q\left(v_{n}\right)$
is $SO\left(Q\right)^{n}$ and this seems too small to obtain the
required statements. If one had access to a precise quantitative equidistribution
statement, in the form of an explicit rate for the limit in (\ref{eq:into 2}),
one could hope to prove results like Theorems \ref{thm:Pair correlation}
and \ref{thm:Pair correlation-1}. Unfortunately, since the the results
of \cite{MR1609447} relied on the equidistribution of unipotent flows,
no good error term was available. However, recently, F. G\"otze and
G. Margulis proved such a statement in the preprint \cite{G-M} (see
also \cite{2010arXiv1004.5123G} for an older version). Their methods
do not rely on the equidistribution of unipotent flows. Instead they
use Fourier analysis to reduce the problem to one of obtaining asymptotic
estimates for certain theta series. In order to estimate these theta
series, they use some of the techniques developed in \cite{MR1609447},
in particular the crux of their proof relies on a non divergence statement
about average of the translates of orbits of certain compact subgroups
in the space of lattices. One cannot apply the results of \cite{G-M}
directly, since in order to do this one would need the error to be
uniform across all intervals. However, the proofs of Theorems \ref{thm:Pair correlation}
and \ref{thm:Pair correlation-1} are based on the methods of \cite{G-M}. 

The object of interest is 
\[
R\left(\mathbb{1}_{P_{Q}^{n}\left(I_{1},\dots,I_{n-1},T\right)}\right)=\sum_{v\in\mathbb{Z}^{nd}}\mathbb{1}_{P_{Q}^{n}\left(I_{1},\dots,I_{n-1},T\right)}\left(v\right)-\int_{\mathbb{R}^{nd}}\mathbb{1}_{P_{Q}^{n}\left(I_{1},\dots,I_{n-1},T\right)}\left(v\right)dv,
\]
where, here and throughout the rest of the paper, for any set $S$,
$\mathbb{1}_{S}$ stands for the characteristic function of the set
$S$. Theorems \ref{thm:Pair correlation} and \ref{thm:Pair correlation-1}
follow from suitable bounds for $\bigl|R\bigl(\mathbb{1}_{P_{Q}^{n}\left(I_{1},\dots,I_{n-1},T\right)}\bigr)\bigr|$.
To obtain these bounds, the function $\mathbb{1}_{P_{Q}^{n}\left(I_{1},\dots,I_{n-1},T\right)}$
is replaced with a smoothened version at the cost of `smoothing errors'
which can be estimated in terms of volumes of certain regions of $\mathbb{R}^{nd}$.
This is carried out in subsections \ref{sub:Smoothing.} and \ref{sub:Volume-estimates.}.
The next step is to use Fourier analysis to transfer the problem into
the `frequency domain'. After taking Fourier transforms, the smoothened
version of $\bigl|R\bigl(\mathbb{1}_{P_{Q}^{n}\left(I_{1},\dots,I_{n-1},T\right)}\bigr)\bigr|$
can be estimated by considering an integral over the `frequency domain',
$\omega\in\mathbb{R}^{n-1}$, of the difference between a theta series,
$\theta\left(\omega\right)$ and its corresponding smooth version,
$\vartheta\left(\omega\right)$ (see (\ref{eq:forgot})). This step
is carried out in subsection \ref{sub:Fourier-transforms}. In order
to estimate the integral, the domain of integration is split into
two parts, namely a neighbourhood of the origin and its complement. 

The integral over the region bounded away from the origin is dealt
with by considering the integral of $\theta\left(\omega\right)$ and
the integral of $\vartheta\left(\omega\right)$ separately. The integral
of $\theta\left(\omega\right)$ contributes the main term in the bound
for $\bigl|R\bigl(\mathbb{1}_{P_{Q}^{n}\left(I_{1},\dots,I_{n-1},T\right)}\bigr)\bigr|$
and it contains the arithmetic information about $Q$. This term is
dealt with in subsection \ref{sub:Main-term}. The integral of $\vartheta\left(\omega\right)$
only contributes a lower order term to the bound for $\bigl|R\bigl(\mathbb{1}_{P_{Q}^{n}\left(I_{1},\dots,I_{n-1},T\right)}\bigr)\bigr|$
and is dealt with in subsection \ref{sub:Lower-order-terms}. These
two integrals can be estimated using techniques and results from \cite{G-M}.
The reason for this, is that $\theta\left(\omega\right)$ and $\vartheta\left(\omega\right)$
can be written as a product of $n$ sums/integrals of the form studied
in \cite{G-M} (see (\ref{eq:product formula rough theta}) and (\ref{eq:product formula for smooth theta})). 

The integral, over the neighbourhood of the origin, is dealt with
in subsection \ref{sub:Lower-order-term2}. This term contributes
a lower order term, but it grows with $n$, faster than the main term.
For $n>d$ this term dominates the main term, explaining why the assumption
$n\leq d$ is needed. The reason for this, is that here we consider
the difference, $\theta\left(\omega\right)-\vartheta\left(\omega\right)$.
Poisson summation is used to convert this into a sum over $\mathbb{Z}^{nd}\setminus\left\{ 0\right\} $,
the problem that arises is that for $m\in\mathbb{Z}^{nd}\setminus\left\{ 0\right\} $
we can still have $m_{i}=0$ for some $1\leq i\leq d$. Therefore,
although it is still possible to take advantage of the fact that the
sum obtained by Poisson summation can be written as a product of $n$
sums, an additional argument is needed to deal with the fact that
$0$ could be included in each of the sums in the product. 

Finally in Section \ref{sec:Proof-of-main} all of the bounds are
collected and Theorems \ref{thm:Pair correlation} and \ref{thm:Pair correlation-1}
are proved. The bounds obtained in Section \ref{sec:Bounding-the-integrals}
depend on the $L^{1}$ norm of a certain function which depends on
a smoothing parameter. In order to prove Theorem \ref{thm:Pair correlation-1}
we need a precise estimates for this norm in terms of the smoothing
parameter. This is carried out in subsection \ref{sub:Norms}.

\section{\label{sec:Set-up.}Set up.}

For the rest of the paper let $n$ and $d$ be natural numbers with
$n\leq d$. In the case when $n=2$, there is only one quadratic form
and the conclusions of Theorems \ref{thm:Pair correlation} and \ref{thm:Pair correlation-1}
follow from the results of \cite{MR1609447} and \cite{G-M}. Hence,
throughout the rest of the paper we suppose that $n\geq3$. For $1\leq i\leq n-1$,
fix intervals $I_{i}$ and $Q:\mathbb{R}^{d}\rightarrow\mathbb{R}$
a nondegenerate quadratic form, suppose that $d\geq5$ and keep the
notation from the introduction. Let $Q_{+}=Q^{2}$, hence $Q_{+}$
corresponds to a positive definite quadratic form. Let $\mathrm{sp}\left(Q\right)$
denote the spectrum of $Q$, $\lambda_{\textrm{min}}=\min_{\lambda\in\mathrm{sp}\left(Q\right)}\left|\lambda\right|$
and $\lambda_{\textrm{max}}=\max_{\lambda\in\mathrm{sp}\left(Q\right)}\left|\lambda\right|$.
Since the problem is unaffected by rescaling $Q$, we may suppose
that $\lambda_{\min}\geq n-1$, this supposition will be used in the
proof of Lemma \ref{lem:Bound on I2}. Define $B\left(T\right)=\left\{ v\in\mathbb{R}^{nd}:\left\Vert v\right\Vert \leq T\right\} $
and $B_{\infty}\left(T\right)=\left\{ v\in\mathbb{R}^{nd}:\left\Vert v\right\Vert _{\infty}\leq T\right\} $,
where we use $\left\Vert .\right\Vert $ to denote the Euclidean norm
and $\left\Vert .\right\Vert _{\infty}$ to denote the maximum norm.
Let 
\[
P_{Q}^{n}\left(I_{1},\dots,I_{n-1}\right)=\left\{ v\in\mathbb{R}^{nd}:Q\left(v_{1}\right)-Q\left(v_{2}\right)\in I_{1},\dots,Q\left(v_{n-1}\right)-Q\left(v_{n}\right)\in I_{n-1}\right\} .
\]
Note that $P_{Q}^{n}\left(I_{1},\dots,I_{n-1},T\right)=P_{Q}^{n}\left(I_{1},\dots,I_{n-1}\right)\cap B\left(T\right)$.
As is standard, we use the notation $\hat{f}$ to denote the Fourier
transform of a function $f$. We will also make heavy use the Vinogradov
asymptotic notation $f\left(s\right)\ll g\left(s\right)$, which means
that there exists some constant $C>0$ such that $f\left(s\right)\leq Cg\left(s\right)$
for all values of $s$ indicated. The constant $C$ will be independent
of those parameters but will usually depend on $d,n,Q$ and the intervals
$I_{1},\dots,I_{n-1}$.

\subsection{\label{sub:Smoothing.}Smoothing.}

For any $i\in\mathbb{N}$, let $k^{i}=k^{i}\left(v\right)dv$ be a
probability measure on $\mathbb{R}^{i}$ with the properties that
it is symmetric around $0$, $k^{i}\left(\left\{ v\in\mathbb{R}^{i}:\left\Vert v\right\Vert \leq1\right\} \right)=1$
and 
\begin{equation}
\bigl|\widehat{k^{i}}\left(v\right)\bigr|\leq\exp\left(-c\sqrt{\left\Vert v\right\Vert }\right)\label{eq:bound on point measure}
\end{equation}
for some positive constant $c$ and all $v\in\mathbb{R}^{i}$. For
any $\tau>0$, let $k_{\tau}^{i}$ denote the rescaled measure such
that $k_{\tau}^{i}\left(A\right)=k^{i}\left(\tau^{-1}A\right)$ for
any measurable set $A$. Note that (\ref{eq:bound on point measure})
implies that 
\begin{equation}
\bigl|\widehat{k_{\tau}^{i}}\left(v\right)\bigr|\leq\exp\left(-c\sqrt{\tau\left\Vert v\right\Vert }\right).\label{eq:bound on smooth point measure}
\end{equation}
For an interval $I=\left[a,b\right]$ and $\epsilon\in\mathbb{R}$,
define $I^{\epsilon}=\left[a-\epsilon,b+\epsilon\right]$. For any
$\tau>0$, $T>0$ and $v\in\mathbb{R}^{nd}$, let 
\[
w_{\pm\tau}\left(v\right)=\mathbb{1}_{B\left(1\pm\tau\right)}*k_{\tau}^{nd}\left(v\right)\quad\textrm{ and }\quad w_{\pm\tau,T}\left(v\right)=w_{\pm\tau}\left(T^{-1}v\right).
\]
For any $\epsilon>0$ and $\omega\in\mathbb{R}^{n-1}$, let 
\[
S_{\pm\epsilon}\left(\omega\right)=\mathbb{1}_{I_{1}^{\pm\epsilon}\times\dots\times I_{n-1}^{\pm\epsilon}}*k_{\epsilon}^{\Pi}\left(\omega\right),
\]
where $k_{\epsilon}^{\Pi}\left(\omega\right)=k_{\epsilon}^{1}\left(\left\langle \omega,e_{1}\right\rangle \right)\dots k_{\epsilon}^{1}\left(\left\langle \omega,e_{n-1}\right\rangle \right)$.
For $v\in\mathbb{R}^{nd}$, let 
\[
S_{\pm\epsilon}^{Q}\left(v\right)=S_{\pm\epsilon}\left(Q\left(v_{1}\right)-Q\left(v_{2}\right),\dots,Q\left(v_{n-1}\right)-Q\left(v_{n}\right)\right).
\]
For a measurable function $f$ on $\mathbb{R}^{nd}$ define 
\begin{equation}
R\left(f\right)=\sum_{v\in\mathbb{Z}^{nd}}f\left(v\right)-\int_{\mathbb{R}^{nd}}f\left(v\right)dv.\label{eq:def of R}
\end{equation}
Note that $R\left(f\right)$ is only well defined if both the quantities
on the right hand side of (\ref{eq:def of R}) are finite. Let $\nu_{T}$
and $\nu_{\tau,T}$ be measures on $\mathbb{R}^{nd}$ and $\mathbb{R}^{n-1}$
respectively, defined by 
\[
\int_{\mathbb{R}^{nd}}fd\nu_{T}=\int_{\mathbb{R}^{nd}}f\left(T^{-1}v\right)\mathbb{1}_{P_{Q}^{n}\left(I_{1},\dots,I_{n-1}\right)}\left(v\right)dv
\]
 and 
\[
\int_{\mathbb{R}^{n-1}}fd\nu_{\tau,T}=\int_{\mathbb{R}^{nd}}f\left(Q\left(v_{1}\right)-Q\left(v_{2}\right),\dots,Q\left(v_{n-1}\right)-Q\left(v_{n}\right)\right)w_{\pm\tau,T}\left(v\right)dv.
\]
In the next two Lemmas we approximate $R\left(\mathbb{1}_{P_{Q}^{n}\left(I_{1},\dots,I_{n-1}\right)}\mathbb{1}_{B\left(T\right)}\right)$
by a smoothened version.
\begin{lem}
\label{Cor::smooth approximation 1}For all $\tau>0$ and $T>0$,
\[
\left|R\left(\mathbb{1}_{P_{Q}^{n}\left(I_{1},\dots,I_{n-1}\right)}\mathbb{1}_{B\left(T\right)}\right)\right|\leq\max_{\pm\tau}\left|R\left(\mathbb{1}_{P_{Q}^{n}\left(I_{1},\dots,I_{n-1}\right)}w_{\pm\tau,T}\right)\right|+\int_{\mathbb{R}^{nd}}\left(\mathbb{1}_{B\left(1+2\tau\right)}-\mathbb{1}_{B\left(1-2\tau\right)}\right)d\nu_{T}.
\]
\end{lem}
\begin{proof}
Define a measure $\mu_{T}$ on $\mathbb{R}^{nd}$, by 
\[
\int_{\mathbb{R}^{nd}}fd\mu_{T}=\sum_{v\in\mathbb{Z}^{nd}}f\left(T^{-1}v\right)\mathbb{1}_{P_{Q}^{n}\left(I_{1},\dots,I_{n-1}\right)}\left(v\right).
\]
Define functions on $\mathbb{R}^{nd}$ by $f=\mathbb{1}_{B\left(1\right)}$
and $f_{\pm\tau}=\mathbb{1}_{B\left(1\pm\tau\right)}$. Note that
\begin{equation}
\left|R\left(\mathbb{1}_{P_{Q}^{n}\left(I_{1},\dots,I_{n-1}\right)}\mathbb{1}_{B\left(T\right)}\right)\right|=\left|\int_{\mathbb{R}^{nd}}fd\left(\mu_{T}-\nu_{T}\right)\right|.\label{eq:smooth approx 1}
\end{equation}
From the definition of $k_{\tau}^{nd}$ it follows that 
\[
f_{-2\tau}\leq f_{-\tau}*k_{\tau}^{nd}\leq f\leq f_{+\tau}*k_{\tau}^{nd}\leq f_{+2\tau}.
\]
Since all the functions in the previous inequality are bounded and
have compact support and the measure $\mu_{T}-\nu_{T}$ is locally
finite, by integrating with respect to $\mu_{T}-\nu_{T}$ we obtain
\begin{alignat*}{1}
\int_{\mathbb{R}^{nd}}fd\left(\mu_{T}-\nu_{T}\right) & \leq\int_{\mathbb{R}^{nd}}f_{+\tau}*k_{\tau}^{nd}d\left(\mu_{T}-\nu_{T}\right)+\int_{\mathbb{R}^{nd}}\left(f_{+\tau}*k_{\tau}^{nd}-f\right)d\nu_{T}\\
 & \leq\int_{\mathbb{R}^{nd}}f_{+\tau}*k_{\tau}^{nd}d\left(\mu_{T}-\nu_{T}\right)+\int_{\mathbb{R}^{nd}}\left(f_{+2\tau}-f_{-2\tau}\right)d\nu_{T}.
\end{alignat*}
Similarly 
\[
\int_{\mathbb{R}^{nd}}fd\left(\mu_{T}-\nu_{T}\right)\geq\int_{\mathbb{R}^{nd}}f_{-\tau}*k_{\tau}^{nd}d\left(\mu-\nu\right)-\int_{\mathbb{R}^{nd}}\left(f_{+2\tau}-f_{-2\tau}\right)d\nu_{T}.
\]
In view of (\ref{eq:smooth approx 1}) and the definition of $w_{\pm\tau,T}$
the conclusion of the Lemma follows from the previous two inequalities. \end{proof}
\begin{lem}
\label{cor:smooth approx 2}For all $\epsilon>0$, $\tau>0$ and $T>0$,
\[
\left|R\left(\mathbb{1}_{P_{Q}^{n}\left(I_{1},\dots,I_{n-1}\right)}w_{\pm\tau,T}\right)\right|\leq\max_{\pm\epsilon}\left|R\left(S_{\pm\epsilon}^{Q}w_{\pm\tau,T}\right)\right|+\int_{\mathbb{R}^{n-1}}\left(\mathbb{1}_{I_{1}^{2\epsilon}\times\dots\times I_{n-1}^{2\epsilon}}-\mathbb{1}_{I_{1}^{-2\epsilon}\times\dots\times I_{n-1}^{-2\epsilon}}\right)d\nu_{\tau,T}.
\]
\end{lem}
\begin{proof}
Define a measure $\mu_{\tau,T}$ on $\mathbb{R}^{n-1}$, by 
\[
\int_{\mathbb{R}^{n-1}}fd\mu_{\tau,T}=\sum_{v\in\mathbb{Z}^{nd}}f\left(Q\left(v_{1}\right)-Q\left(v_{2}\right),\dots,Q\left(v_{n-1}\right)-Q\left(v_{n}\right)\right)w_{\pm\tau,T}\left(v\right).
\]
Define functions on $\mathbb{R}^{nd}$ by $f=\mathbb{1}_{I_{1}\times\dots\times I_{n-1}}$
and $f_{\pm\epsilon}=\mathbb{1}_{I_{1}^{\pm\epsilon}\times\dots\times I_{n-1}^{\pm\epsilon}}$.
Note that 
\begin{equation}
\left|R\left(\mathbb{1}_{P_{Q}^{n}\left(I_{1},\dots,I_{n-1}\right)}w_{\pm\tau,T}\right)\right|=\left|\int_{\mathbb{R}^{n-1}}fd\left(\mu_{\tau,T}-\nu_{\tau,T}\right)\right|.\label{eq:smooth approx 1-1}
\end{equation}
From the definition of $k_{\epsilon}^{\Pi}$ it follows that 
\[
f_{-2\epsilon}\left(\omega\right)\leq f_{-\epsilon}*k_{\epsilon}^{\Pi}\left(\omega\right)\leq f\left(\omega\right)\leq f_{+\epsilon}*k_{\epsilon}^{\Pi}\left(\omega\right)\leq f_{+2\epsilon}.
\]
Since all the functions in the previous inequality are bounded and
have compact support and the measure $\mu_{\tau,T}-\nu_{\tau,T}$
is locally finite, by integrating with respect to $\mu_{\tau,T}-\nu_{\tau,T}$
as in the proof of Lemma \ref{Cor::smooth approximation 1} we obtain
\begin{multline*}
\int_{\mathbb{R}^{n-1}}f_{-\epsilon}*k_{\epsilon}^{\Pi}d\left(\mu_{\tau,T}-\nu_{\tau,T}\right)-\int_{\mathbb{R}^{n-1}}\left(f_{+2\epsilon}-f_{-2\epsilon}\right)d\nu_{\tau,T}\\
\leq\int_{\mathbb{R}^{n-1}}fd\left(\mu_{\tau,T}-\nu_{\tau,T}\right)\\
\leq\int_{\mathbb{R}^{n-1}}f_{+\epsilon}*k_{\epsilon}^{\Pi}d\left(\mu_{\tau,T}-\nu_{\tau,T}\right)+\int_{\mathbb{R}^{n-1}}\left(f_{+2\epsilon}-f_{-2\epsilon}\right)d\nu_{\tau,T}.
\end{multline*}
In view of (\ref{eq:smooth approx 1-1}) and the definition of $S_{\pm\epsilon}^{Q}$
the conclusion of the Lemma follows from the previous inequality. 
\end{proof}
In subsection \ref{sub:Volume-estimates.} we will obtain bounds for
the smoothing errors 
\[
\int_{\mathbb{R}^{nd}}\left(\mathbb{1}_{B\left(1+2\tau\right)}-\mathbb{1}_{B\left(1-2\tau\right)}\right)d\nu_{T}\quad\textrm{and}\quad\int_{\mathbb{R}^{n-1}}\left(\mathbb{1}_{I_{1}^{2\epsilon}\times\dots\times I_{n-1}^{2\epsilon}}-\mathbb{1}_{I_{1}^{-2\epsilon}\times\dots\times I_{n-1}^{-2\epsilon}}\right)d\nu_{\tau,T}
\]
that arise from Lemmas \ref{Cor::smooth approximation 1} and \ref{cor:smooth approx 2}.

\subsection{\label{sub:Fourier-transforms}Fourier transforms }

To obtain bounds for $\left|R\left(S_{\pm\epsilon}^{Q}w_{\pm\tau,T}\right)\right|$
the strategy of \cite{G-M} will be used, in particular we proceed
via the Fourier transform. Numerous text books on Fourier analysis
are available, for instance see \cite{MR2445437}. Let $\mathcal{S}\left(\mathbb{R}^{d}\right)$
denote the space of Schwartz functions on $\mathbb{R}^{d}$ (see Section
2.2 of \cite{MR2445437}) and let $\mathcal{C}_{0}^{\infty}\left(\mathbb{R}^{d}\right)$
be smooth functions with compact support on $\mathbb{R}^{d}$. We
note that $\mathcal{C}_{0}^{\infty}\left(\mathbb{R}^{d}\right)\subset\mathcal{S}\left(\mathbb{R}^{d}\right)\subset L^{1}\left(\mathbb{R}^{d}\right)$
and that $\mathcal{S}\left(\mathbb{R}^{d}\right)$ is invariant under
the Fourier transform. For $v\in\mathbb{R}^{nd}$, let $Q_{++}\left(v\right)=\sum_{i=1}^{n}Q_{+}\left(v_{i}\right)$
and $\zeta_{\pm\tau}\left(v\right)=w_{\pm\tau}\left(v\right)\exp\left(Q_{++}\left(v\right)\right)$.
Note that if $f\in\mathcal{C}_{0}^{\infty}\left(\mathbb{R}^{d}\right)$
and $g=\mathbb{1}_{A}$, where $A$ is a compact subset of $\mathbb{R}^{d}$
then $g*f\in\mathcal{C}_{0}^{\infty}\left(\mathbb{R}^{d}\right)$.
This fact implies that $S_{\epsilon}\in\mathcal{S}\left(\mathbb{R}^{n-1}\right)$
and $\zeta_{\pm\tau}\in\mathcal{S}\left(\mathbb{R}^{nd}\right)$,
therefore it is possible to use the Fourier inversion formula. Hence
\begin{equation}
S_{\pm\epsilon}^{Q}\left(v\right)=\left(2\pi\right)^{1-n}\int_{\mathbb{R}^{n-1}}\widehat{S}_{\pm\epsilon}\left(\omega\right)\exp\left(i\left\langle \left(Q\left(v_{1}\right)-Q\left(v_{2}\right),\dots,Q\left(v_{n-1}\right)-Q\left(v_{n}\right)\right),\omega\right\rangle \right)d\omega\label{eq:inversion of S epsilon}
\end{equation}
and 
\begin{equation}
\zeta_{\pm\tau}\left(v\right)=\left(2\pi\right)^{-nd}\int_{\mathbb{R}^{nd}}\widehat{\zeta}_{\pm\tau}\left(y\right)\exp\left(i\left\langle v,y\right\rangle \right)dy.\label{eq:inversion of zeta}
\end{equation}
Therefore, by using the definition (\ref{eq:def of R}) and (\ref{eq:inversion of S epsilon})
we obtain 
\begin{equation}
R\left(S_{\pm\epsilon}^{Q}w_{\pm\tau,T}\right)=\left(2\pi\right)^{1-n}\int_{\mathbb{R}^{n-1}}R\left(e_{Q,\omega}w_{\pm\tau,T}\right)\widehat{S}_{\pm\epsilon}\left(\omega\right)d\omega,\label{eq: error1}
\end{equation}
where $e_{Q,\omega}\left(v\right)=\exp\left(i\left\langle \left(Q\left(v_{1}\right)-Q\left(v_{2}\right),\dots,Q\left(v_{n-1}\right)-Q\left(v_{n}\right)\right),\omega\right\rangle \right)$.
Also using the definitions of $w_{\pm\tau,T}$ and $\zeta_{\pm\tau}$
we have
\[
R\left(e_{Q,\omega}w_{\pm\tau,T}\right)=R\left(e_{Q,\omega}\left(v\right)\zeta_{\pm\tau}\left(v/T\right)\exp\left(-Q_{++}\left(v/T\right)\right)\right).
\]
Combining this with (\ref{eq:inversion of zeta}) gives 
\begin{alignat*}{1}
R\left(e_{Q,\omega}w_{\pm\tau,T}\right) & =\left(2\pi\right)^{-nd}\int_{\mathbb{R}^{nd}}R\left(e_{Q,\omega}\left(v\right)\exp\left(2\pi i\left\langle v/T,y\right\rangle -Q_{++}\left(v/T\right)\right)\right)\widehat{\zeta}_{\pm\tau}\left(y\right)dy\\
 & =\left(2\pi\right)^{-nd}\int_{\mathbb{R}^{nd}}R\left(e_{Q,\omega}\tilde{e}_{Q,T,y}\right)\widehat{\zeta}_{\pm\tau}\left(y\right)dy
\end{alignat*}
where $\tilde{e}_{Q,T,y}\left(v\right)=\exp\left(i\left\langle v/T,y\right\rangle -Q_{++}\left(v/T\right)\right)$.
We now write 
\[
R\left(e_{Q,\omega}\tilde{e}_{Q,T,y}\right)=\theta_{T,y}\left(\omega\right)-\vartheta_{T,y}\left(\omega\right),
\]
where $\theta_{T,y}\left(\omega\right)$ and $\vartheta_{T,y}\left(\omega\right)$
are defined as follows: For $x\in\mathbb{R}$, let $Q_{T,y}\left(x,v_{i}\right)=i\left(xQ\left(v_{i}\right)+\left\langle v_{i}/T,y\right\rangle \right)-T^{-2}Q_{+}\left(v_{i}\right)$
and for $\omega\in\mathbb{R}^{n-1}$ define 
\[
\overline{Q}_{T,y}\left(\omega,v\right)=\sum_{i=1}^{n}Q_{T,y}\left(\omega_{i}-\omega_{i-1},v_{i}\right),
\]
where $\omega_{0}=\omega_{n}=0$. 
\begin{rem}
For the rest of the paper the convention that $\omega_{0}=\omega_{n}=0$,
will be used in order to simplify the notation. 
\end{rem}
For $\omega\in\mathbb{R}^{n-1}$ and $y\in\mathbb{R}^{nd}$, let 
\begin{equation}
\theta_{T,y}\left(\omega\right)=\sum_{v\in\mathbb{Z}^{nd}}\exp\left(\overline{Q}_{T,y}\left(\omega,v\right)\right)\quad\textrm{and}\quad\vartheta_{T,y}\left(\omega\right)=\int_{\mathbb{R}^{nd}}\exp\left(\overline{Q}_{T,y}\left(\omega,v\right)\right)dv,\label{eq:forgot}
\end{equation}
From (\ref{eq:forgot}) and the definition of $\overline{Q}_{T,y}$,
it follows that 
\begin{equation}
\theta_{T,y}\left(\omega\right)=\prod_{i=1}^{n}\sum_{v_{i}\in\mathbb{Z}^{d}}\exp\left(Q_{T,y}\left(\omega_{i}-\omega_{i-1},v_{i}\right)\right)\label{eq:product formula rough theta}
\end{equation}
and 
\begin{equation}
\vartheta_{T,y}\left(\omega\right)=\prod_{i=1}^{n}\int_{\mathbb{R}^{d}}\exp\left(Q_{T,y}\left(\omega_{i}-\omega_{i-1},v_{i}\right)\right)dv_{i}.\label{eq:product formula for smooth theta}
\end{equation}
Next we define a certain bounded region of $\mathbb{R}^{n-1}$. Let
\[
\mathcal{B}\left(T\right)=\left\{ \omega\in\mathbb{R}^{n-1}:\left|\omega_{i}-\omega_{i-1}\right|\leq T^{-1}\textrm{ for }1\leq i\leq n\right\} .
\]
Decomposing the integral over $\mathbb{R}^{n-1}$ in (\ref{eq: error1})
into regions we see that 
\begin{equation}
\left|R\left(S_{\pm\epsilon}^{Q}w_{\pm\tau,T}\right)\right|\ll E_{0}\left(\pm\tau,\pm\epsilon,T\right)+E_{1}\left(\pm\tau,\pm\epsilon,T\right)+E_{2}\left(\pm\tau,\pm\epsilon,T\right)\label{eq:decomposition}
\end{equation}
 where 
\begin{alignat*}{1}
E_{0} & \left(\pm\tau,\pm\epsilon,T\right)=\left|\int_{\mathbb{R}^{n-1}\setminus\mathcal{B}\left(T\right)}\widehat{S}_{\pm\epsilon}\left(\omega\right)\int_{\mathbb{R}^{nd}}\vartheta_{T,y}\left(\omega\right)\widehat{\zeta}_{\pm\tau}\left(y\right)dyd\omega\right|\\
E_{1} & \left(\pm\tau,\pm\epsilon,T\right)=\left|\int_{\mathcal{B}\left(T\right)}R\left(e_{Q,t}w_{\pm\tau,T}\right)\widehat{S}_{\pm\epsilon}\left(\omega\right)d\omega\right|\\
E_{2} & \left(\pm\tau,\pm\epsilon,T\right)=\left|\int_{\mathbb{R}^{n-1}\setminus\mathcal{B}\left(T\right)}\widehat{S}_{\pm\epsilon}\left(\omega\right)\int_{\mathbb{R}^{nd}}\theta_{T,y}\left(\omega\right)\widehat{\zeta}_{\pm\tau}\left(y\right)dyd\omega\right|.
\end{alignat*}

\section{\label{sec:Bounding-the-integrals}Bounding the integrals}

In this section we obtain bounds for the integrals $E_{0}$, $E_{1}$
and $E_{2}$, in terms of $\bigl\Vert\widehat{\zeta}_{\tau}\bigr\Vert_{1}$.
Precise bounds for $\bigl\Vert\widehat{\zeta}_{\tau}\bigr\Vert_{1}$
are given in subsection \ref{sub:Norms}. We will only consider the
case when $\epsilon>0$ and $\tau>0$ since the other three cases
can be dealt with in an identical manner. The following Theorem will
be proved in Propositions \ref{lem:Bound on I1-1}, \ref{lem:bound on I0}
and \ref{lem:Bound on I2}.
\begin{thm}
For all $0<\epsilon<1$ and $\tau>0$, there exists $T_{0}>0$ such
that for all $T>T_{0}$,
\begin{alignat*}{1}
E_{0}\left(\tau,\epsilon,T\right) & \ll\bigl\Vert\widehat{\zeta}_{\tau}\bigr\Vert_{1}T^{\left(n-1\right)d-2\left(n-2\right)}\\
E_{1}\left(\tau,\epsilon,T\right) & \ll\left(\bigl\Vert\widehat{\zeta}_{\tau}\bigr\Vert_{1}+\tau^{\left(1-nd\right)/2}\right)T^{\left(n-1\right)d-n+1}\\
E_{2}\left(\tau,\epsilon,T\right) & \ll\bigl\Vert\widehat{\zeta}_{\tau}\bigr\Vert_{1}\mathscr{A}_{\epsilon}\left(T\right)T^{nd-2\left(n-1\right)},
\end{alignat*}
where for any fixed $\epsilon>0$ we have $\lim_{T\rightarrow\infty}\mathscr{A}_{\epsilon}\left(T\right)=0$
provided that $Q$ is irrational. (See (\ref{eq:def of A_epT}) for
a precise definition of $\mathscr{A}_{\epsilon}\left(T\right)$.)
\end{thm}
The bounds for $E_{0}$ and $E_{1}$ contribute only to lower order
terms. Note that the bound for $E_{0}$ is of smaller order of magnitude
than $T^{nd-2\left(n-1\right)}$ all $n\in\mathbb{N}$. The bound
for $E_{1}$ is of smaller order of magnitude than $T^{nd-2\left(n-1\right)}$
only for $n\leq d$. Using the fact that $\vartheta_{y}\left(t\right)$
can be split as in (\ref{eq:product formula for smooth theta}) the
required bound for $E_{0}$ is relatively simple to obtain. The bound
for $E_{1}$ is slightly more involved since the formula (\ref{eq:poissoin summation})
is used. This means that, although one can still take advantage of
the splitting given in (\ref{eq:product formula rough theta}), the
sums in the product may include $0$ and this causes extra difficulties.
To overcome these difficulties we employ Lemma \ref{technical lemma},
which enables us to bound the minimum of certain quantities from below
by a weighted average. The bound for $E_{2}$ contributes to the main
term and this term depends on the arithmetic properties of $Q$. Using
(\ref{eq:product formula rough theta}), the bound for $E_{2}$ follows
reasonably directly from results in \cite{G-M}.

\subsection{\label{sub:Lower-order-terms}Bound for $E_{0}$.}

The following bound will be used in subsections \ref{sub:Lower-order-terms}
and \ref{sub:Lower-order-term2} to obtain bounds for $E_{0}$ and
$E_{1}$. It is relatively straightforward to prove via a direct computation
involving Gaussian integrals (see Formula (3.28) in \cite{G-M}).
The notation $Q_{+}^{-1}$ will stand for the positive definite quadratic
form that corresponds to the matrix $Q_{+}^{-1}$. 
\begin{lem}
\label{lem:bound on smooth theta}For $1\leq i\leq n-1$, all $y\in\mathbb{R}^{nd}$,
$x\in\mathbb{R}$ and $T>0$,
\[
\left|\int_{\mathbb{R}^{d}}\exp\left(Q_{T,y}\left(x,v_{i}\right)\right)dv_{i}\right|\ll g_{T}\left(x\right)T^{d/2}\exp\left(-\frac{1}{4}g_{T}\left(x\right)Q_{+}^{-1}\left(y_{i}/T\right)\right),
\]
where $g_{T}\left(x\right)=T/\sqrt{1+\left(xT^{2}\right)^{2}}$.
\end{lem}
We will need estimates for the Fourier transform of the smoothened
characteristic function. 
\begin{lem}
\label{lem:bound on smoothed indictaor of P}For all $\epsilon>0$
and $\omega\in\mathbb{R}^{n-1}$, 
\[
\bigl|\widehat{S}_{\epsilon}\left(\omega\right)\bigr|\ll\prod_{i=1}^{n-1}\min\left\{ 1,\left|\frac{1}{\omega_{i}}\right|\right\} \exp\left(-c\sqrt{\left|\epsilon\omega_{i}\right|}\right).
\]
\end{lem}
\begin{proof}
Using the definition of $S_{\epsilon}$ we get $\bigl|\widehat{S}_{\epsilon}\left(\omega\right)\bigr|\leq\bigl|\mathbb{\widehat{1}}_{I_{1}^{\epsilon}\times\dots\times I_{n-1}^{\epsilon}}\left(\omega\right)\bigr|\bigl|\widehat{k_{\epsilon}^{1}}\left(\omega_{1}\right)\dots\widehat{k_{\epsilon}^{1}}\left(\omega_{n-1}\right)\bigr|$.
Then a simple computation and (\ref{eq:bound on point measure}) gives
\begin{alignat*}{1}
\bigl|\widehat{S}_{\epsilon}\left(\omega\right)\bigr| & \ll\prod_{i=1}^{n-1}\left|\frac{1}{\omega_{i}}\sin\left(\omega_{i}\pi\bigl|I_{i}\bigr|\right)\right|\exp\left(-c\sqrt{\left|\epsilon\omega_{i}\right|}\right).
\end{alignat*}
Since $\left|\frac{\sin(kx)}{x}\right|\leq\min\left\{ k,\left|1/x\right|\right\} $
for all $x\in\mathbb{R}$, the claim of the Lemma follows.
\end{proof}
The bound for $E_{0}$ is obtained by using Lemmas \ref{lem:bound on smooth theta}
and \ref{lem:bound on smoothed indictaor of P}, together with some
elementary estimates of integrals of powers of $g_{T}\left(x\right)$.
For $\omega\in\mathbb{R}^{n-1}$, define $G_{T}\left(\omega\right)=\prod_{i=1}^{n}g_{T}\left(\omega_{i}-\omega_{i-1}\right)$. 
\begin{prop}
\label{lem:Bound on I1-1}For all $\epsilon>0$, $\tau>0$ and $T>0$,
\textup{
\[
E_{0}\left(\tau,\epsilon,T\right)\ll\bigl\Vert\widehat{\zeta}_{\tau}\bigr\Vert_{1}T^{\left(n-1\right)d-2\left(n-2\right)}.
\]
}\end{prop}
\begin{proof}
Using (\ref{eq:product formula for smooth theta}) and Lemma \ref{lem:bound on smooth theta},
\[
\left|\vartheta_{T,y}\left(\omega\right)\right|\ll T^{nd/2}G_{T}\left(\omega\right){}^{d/2}.
\]
Note that for any $x\in\mathbb{R}$ and $T\in\mathbb{R}$, $g_{T}\left(x\right)\leq\left|Tx\right|^{-1}$.
It follows that for all $\omega\in\mathbb{R}^{n-1}$ and $y\in\mathbb{R}^{nd}$
and $1\leq i,j\leq n$ with $i\neq j$, 

\begin{equation}
\left|\vartheta_{T,y}\left(\omega\right)\right|\ll T^{\left(n-2\right)d/2}\left(\frac{G_{T}\left(\omega\right)}{g_{T}\left(\omega_{i}-\omega_{i-1}\right)g_{T}\left(\omega_{j}-\omega_{j-1}\right)}\right)^{d/2}\left(\left|\omega_{i}-\omega_{i-1}\right|\left|\omega_{j}-\omega_{j-1}\right|\right)^{-d/2}.\label{eq:lem31}
\end{equation}
Choose $1\leq l\leq n$ with $l\neq i$ or $j$. Since, we assume
that $n\geq3$, this is always possible. Note that $g_{T}\left(x\right)\leq T$
and hence 
\begin{equation}
\left(\frac{G_{T}\left(\omega\right)}{g_{T}\left(\omega_{i}-\omega_{i-1}\right)g_{T}\left(\omega_{j}-\omega_{j-1}\right)}\right)^{d/2}\leq T^{d/2}\widetilde{G}_{T}\left(\omega\right)^{d/2},\label{eq:lem32}
\end{equation}
where
\[
\widetilde{G}_{T}\left(\omega\right)=\frac{G{}_{T}\left(\omega\right)}{g_{T}\left(\omega_{i}-\omega_{i-1}\right)g_{T}\left(\omega_{j}-\omega_{j-1}\right)g_{T}\left(\omega_{l}-\omega_{l-1}\right)}.
\]
Let $T_{n}=\left(n-1\right)T$ and 
\[
\mathcal{B}_{1}=\underset{i\neq j}{\bigcup_{1\leq i,j\leq n}}\mathcal{B}_{i,j},\textrm{ where }\mathcal{B}_{i,j}=\left\{ \omega\in\mathbb{R}^{n-1}:\left|\omega_{i}-\omega_{i-1}\right|>T_{n}^{-1},\left|\omega_{j}-\omega_{j-1}\right|>T_{n}^{-1}\right\} .
\]
Note that $\mathbb{R}^{n-1}\setminus\mathcal{B}\left(T\right)\subset\mathcal{B}_{1}$.
(Or, equivalently, $\mathbb{R}^{n-1}\setminus\mathcal{B}_{1}\subset\mathcal{B}\left(T\right)$.)
To see this, suppose that $\omega\in\mathbb{R}^{n-1}\setminus\mathcal{B}_{1}$,
then $\left|\omega_{i}-\omega_{i-1}\right|>T_{n}^{-1}$ for at most
one $1\leq i\leq n$. If $\left|\omega_{i}-\omega_{i-1}\right|\leq T_{n}^{-1}$
for all $1\leq i\leq n$, then clearly $\omega\in\mathcal{B}\left(T\right)$.
Suppose $1\leq l\leq n$ is such that $\left|\omega_{l}-\omega_{l-1}\right|>T_{n}^{-1}$.
Now note that $\left|\omega_{l}-\omega_{l-1}\right|=\Bigl|\sum_{1\leq i\leq n,i\neq l}\left(\omega_{i}-\omega_{i-1}\right)\Bigr|\leq T^{-1}$
by the triangle inequality and hence $\omega\in\mathcal{B}\left(T\right)$. 

Using the definition of $E_{0}$, 
\begin{alignat}{1}
E_{0}\left(\tau,\epsilon,T\right) & =\left|\int_{\mathbb{R}^{n-1}\setminus\mathcal{B}\left(T\right)}\widehat{S}_{\epsilon}\left(\omega\right)\int_{\mathbb{R}^{nd}}\vartheta_{T,y}\left(\omega\right)\widehat{\zeta}_{\tau}\left(y\right)dyd\omega\right|\nonumber \\
 & \leq\int_{\mathbb{R}^{n-1}\setminus\mathcal{B}\left(T\right)}\left|\widehat{S}_{\epsilon}\left(\omega\right)\int_{\mathbb{R}^{nd}}\vartheta_{T,y}\left(\omega\right)\widehat{\zeta}_{\tau}\left(y\right)dy\right|d\omega\nonumber \\
 & \leq\underset{i\neq j}{\sum_{1\leq i,j\leq n}}\int_{\mathcal{B}_{i,j}}\left|\widehat{S}_{\epsilon}\left(\omega\right)\int_{\mathbb{R}^{nd}}\vartheta_{T,y}\left(\omega\right)\widehat{\zeta}_{\tau}\left(y\right)dy\right|d\omega.\label{eq:extra e1}
\end{alignat}
From Lemma \ref{lem:bound on smoothed indictaor of P} we get that
for all $\epsilon>0$ we have the uniform bound, $\bigl|\widehat{S}_{\epsilon}\left(\omega\right)\bigr|\ll1$.
Hence using (\ref{eq:lem31}) and (\ref{eq:lem32}), 
\begin{alignat}{1}
\int_{\mathcal{B}_{i,j}}\left|\widehat{S}_{\epsilon}\left(\omega\right)\int_{\mathbb{R}^{nd}}\vartheta_{T,y}\left(\omega\right)\widehat{\zeta}_{\tau}\left(y\right)dy\right|d\omega & \ll\bigl\Vert\widehat{\zeta}_{\tau}\bigr\Vert_{1}T^{\left(n-1\right)d/2}\int_{\mathcal{B}_{i,j}}\widetilde{G}_{T}\left(\omega\right)^{d/2}\left(\left|\omega_{i}-\omega_{i-1}\right|\left|\omega_{j}-\omega_{j-1}\right|\right)^{-d/2}d\omega.\label{eq:lem34}
\end{alignat}
By doing the change of variables $\varphi:\omega_{i}-\omega_{i-1}\rightarrow\begin{cases}
\xi_{i} & \textrm{if }i<l\\
\xi_{i-1} & \textrm{if }i>l
\end{cases}$ we get 
\begin{equation}
\int_{\mathcal{B}_{i,j}}\widetilde{G}_{T}\left(\omega\right)^{d/2}\left(\left|\omega_{i}-\omega_{i-1}\right|\left|\omega_{j}-\omega_{j-1}\right|\right)^{-d/2}d\omega\leq\int_{\varphi\left(\mathcal{B}_{i,j}\right)}\underset{k\neq i,j}{\prod_{1\leq k\leq n-1}}g_{T}\left(\xi_{k}\right)^{d/2}\left(\left|\xi_{i}\right|\left|\xi_{j}\right|\right)^{-d/2}d\xi,\label{eq:lem33}
\end{equation}
where $\varphi\left(\mathcal{B}_{i,j}\right)=\left\{ \xi\in\mathbb{R}^{n-1}:\left|\xi_{i}\right|\geq T_{n}^{-1},\left|\xi_{j}\right|\geq T_{n}^{-1}\right\} $.
Note that 
\[
\int_{T_{n}^{-1}}^{\infty}x^{-d/2}dx\ll T^{d/2-1}.
\]
By making change of variables $T^{2}x=\sinh y$, we get 
\[
\int_{\mathbb{R}}g_{T}\left(x\right)^{d/2}dx=T^{d/2-2}\int_{\mathbb{R}}\frac{1}{\cosh^{d/4-1}y}dy\ll T^{d/2-2}.
\]
The last two observations, (\ref{eq:lem34}) and (\ref{eq:lem33})
imply that for all $1\leq i,j\leq n$ with $i\neq j$, 
\[
\int_{\mathcal{B}_{i,j}}\left|\widehat{S}_{\epsilon}\left(\omega\right)\int_{\mathbb{R}^{nd}}\vartheta_{T,y}\left(\omega\right)\widehat{\zeta}_{\tau}\left(y\right)dy\right|d\omega\leq\bigl\Vert\widehat{\zeta}_{\tau}\bigr\Vert_{1}T^{\left(n-1\right)d-2\left(n-2\right)}.
\]
The conclusion of the Lemma follows from (\ref{eq:extra e1}). 
\end{proof}

\subsection{\label{sub:Lower-order-term2}Bound for $E_{1}$}

We will need two preliminary Lemmas. The first is probably standard,
but for completeness, a proof is provided. 
\begin{lem}
\label{lem:Bound on exp sum}For any $c>0$, there exists a positive
constant $B$ such that, for any $y\in\mathbb{R}^{nd}$ 
\[
\sum_{m\in\mathbb{Z}^{nd}\setminus\left\{ 0\right\} }\exp\left(-c\left\Vert y-m\right\Vert ^{2}\right)<B.
\]
\end{lem}
\begin{proof}
For $v\in\left[-1/2,1/2\right]^{nd}$ , $\left\Vert v\right\Vert ^{2}\leq nd/4$.
Hence 
\[
\int_{\left[-1/2,1/2\right]^{nd}}\exp\left(-c\left\Vert u+v\right\Vert ^{2}\right)dv\geq\exp\left(-c\left\Vert u\right\Vert ^{2}-ndc/4\right)\int_{\left[-1/2,1/2\right]^{nd}}\exp\left(-2c\left\langle u,v\right\rangle \right)dv.
\]
It is easy to check that for any $u\in\mathbb{R}^{nd}$, $\int_{\left[-1/2,1/2\right]^{nd}}\exp\left(-2c\left\langle u,v\right\rangle \right)dv\geq1$
and hence we get the inequality 
\[
\exp\left(-c\left\Vert u\right\Vert ^{2}\right)\ll\int_{\left[-1/2,1/2\right]^{nd}}\exp\left(-c\left\Vert u+v\right\Vert ^{2}\right)dv.
\]
Hence 
\begin{alignat*}{1}
\sum_{m\in\mathbb{Z}^{nd}\setminus\left\{ 0\right\} }\exp\left(-c\left\Vert y-m\right\Vert ^{2}\right) & \ll\sum_{m\in\mathbb{Z}^{nd}\setminus\left\{ 0\right\} }\int_{\left[-1/2,1/2\right]^{nd}}\exp\left(-c\left\Vert y-m+v\right\Vert ^{2}\right)dv\\
 & =\sum_{m\in\mathbb{Z}^{nd}\setminus\left\{ 0\right\} }\int_{\left[-1/2,1/2\right]^{nd}+y-m}\exp\left(-c\left\Vert z\right\Vert ^{2}\right)dz\\
 & \ll\int_{\mathbb{R}^{nd}}\exp\left(-c\left\Vert z\right\Vert ^{2}\right)dz<\infty.
\end{alignat*}

\end{proof}
The second preliminary Lemma is the crucial step in obtaining the
estimate for $E_{1}$. For $\omega\in\mathbb{R}^{n-1}$, let $G{}_{T}^{*}\left(\omega\right)=\min_{1\leq i\leq n}g_{T}\left(\omega_{i}-\omega_{i-1}\right)^{2}$.
\begin{lem}
\label{technical lemma}Let $\omega\in\mathcal{B}\left(T\right)$
and $\left\{ s_{1},\dots,s_{n}\right\} =\left\{ \omega_{i}-\omega_{i-1}\right\} _{1\leq i\leq n}$.
Suppose $s_{1}=\max_{1\leq i\leq n}\left|s_{i}\right|$. Then, there
exist positive constants $c_{1},\dots,c_{n}$, and $b_{2},\dots,b_{n}$
such that, for all $T\geq1$, 
\[
G{}_{T}^{*}\left(\omega\right)\geq\sum_{i=1}^{n}\frac{c_{i}}{T^{b_{i}}}g_{T}\left(s_{i}\right)^{2},
\]
where $b_{1}=0$ and $\sum_{i=2}^{n}b_{i}\leq2\left(n-2\right)$. \end{lem}
\begin{proof}
For $\omega\in\mathcal{B}\left(T\right)$, let $\left\{ s_{1},\dots,s_{n}\right\} =\left\{ \omega_{i}-\omega_{i-1}\right\} _{1\leq i\leq n}$
and $s_{1}=\max_{1\leq i\leq n}\left|s_{i}\right|$. Thus, $G{}_{T}^{*}\left(\omega\right)=g_{T}\left(s_{0}\right)^{2}$.
Let 
\[
\alpha_{i}=\begin{cases}
-\log\left|s_{i}\right|/\log T & \textrm{ if }\left|s_{i}\right|\in\left[1/T^{2},1/T\right]\\
2 & \textrm{ if }\left|s_{i}\right|\in\left[0,1/T^{2}\right].
\end{cases}
\]
Note that $1\leq\alpha_{i}\leq2$ for all $1\leq i\leq n$ and $\alpha_{1}=\min_{1\leq i\leq n}\alpha_{i}$.
For all $s_{i}$ such that $\left|s_{i}\right|\in\left[1/T^{2},1/T\right]$
we have $\left|s_{i}\right|=1/T^{\alpha_{i}}$ . Thus 
\begin{equation}
\left(\frac{g_{T}\left(s_{0}\right)}{g_{T}\left(s_{i}\right)}\right)^{2}=\frac{1+T^{4-2\alpha_{i}}}{1+T^{4-2\alpha_{0}}}\geq\frac{1}{2T^{2\left(\alpha_{i}-\alpha_{1}\right)}}.\label{eq:tech lemma 2}
\end{equation}
For all $s_{i}$ such that $\left|s_{i}\right|\in\left[0,1/T^{2}\right]$
we have $g_{T}\left(s_{i}\right)^{2}\in\left[T^{2}/2,T^{2}\right]$
. Thus 
\begin{equation}
\left(\frac{g_{T}\left(s_{1}\right)}{g_{T}\left(s_{i}\right)}\right)^{2}\geq\frac{1}{1+T^{4-2\alpha_{1}}}\geq\frac{1}{2T^{2\left(\alpha_{i}-\alpha_{1}\right)}}.\label{eq:tech lemma 3}
\end{equation}
Let $\delta_{i}=\alpha_{i}-\alpha_{1}$. Note that for all $1\leq i\leq n-1$,
$0\leq\delta_{i}\leq1$. Moreover, $\sum_{i=1}^{n}s_{i}=0$ and hence
\[
\left|s_{1}\right|\leq\sum_{s_{i}\in\left\{ s_{1},\dots,s_{n}\right\} \setminus\left\{ s_{1}\right\} }\left|s_{i}\right|,
\]
which implies that 
\[
\frac{1}{T^{\alpha_{1}}}\leq\sum_{i=2}^{n-1}\frac{1}{T^{\alpha_{i}}}=\sum_{i=2}^{n-1}\frac{1}{T^{\delta_{i}+\alpha_{1}}}.
\]
Thus 
\begin{equation}
1\leq\sum_{i=2}^{n-1}\frac{1}{T^{\delta_{i}}}.\label{eq:tech lemma 1}
\end{equation}
Let $\sum_{i=2}^{n-1}\delta_{i}=\Delta$ and note that for all $2\leq i\leq n-1$,
$\Delta-\delta_{i}\leq n-2$. Hence by multiplying by the denominators
in (\ref{eq:tech lemma 1}) we get 
\[
T^{\Delta}\leq\sum_{i=2}^{n-1}T^{\Delta-\delta_{i}}\leq\left(n-2\right)T^{n-2}.
\]
Since this holds for all $T\geq1$ we get $\sum_{i=2}^{n-1}\delta_{i}\leq n-2$.
Finally, note that from (\ref{eq:tech lemma 2}) and (\ref{eq:tech lemma 3}),
we obtain 
\[
G{}_{T}^{*}\left(\omega\right)=g_{T}\left(s_{0}\right)^{2}\geq\frac{1}{n}\left(g_{T}\left(s_{1}\right)^{2}+\sum_{s_{i}\in\left\{ s_{1},\dots,s_{n}\right\} \setminus\left\{ s_{1}\right\} }\frac{1}{2T^{2\left(\alpha_{i}-\alpha_{1}\right)}}g_{T}\left(s_{i}\right)^{2}\right),
\]
since $\sum_{s_{i}\in\left\{ s_{1},\dots,s_{n}\right\} \setminus\left\{ s_{1}\right\} }2\left(\alpha_{i}-\alpha_{1}\right)=\sum_{i=2}^{n-1}2\delta_{i}\leq2\left(n-2\right)$
the claim of the Lemma follows. 
\end{proof}
Using the preceding two results we can now obtain a bound for $E_{1}$.
This is done by using Poisson summation to convert the difference
$\theta_{T,y}\left(\omega\right)-\vartheta_{T,y}\left(\omega\right)$,
into a sum over $m\in\mathbb{Z}^{nd}\setminus\left\{ 0\right\} $.
An estimate for each term in the sum can be obtained provided that
$\left\Vert y/T-m\right\Vert $ is bounded from below. The integral
over $y\in\mathbb{R}^{nd}$ is then decomposed into boxes centred
at each point of $\mathbb{Z}^{nd}$. Thus, the estimates for the summands
can be used and there will be additional term coming from the point
at the centre of the box under consideration. Lemma \ref{technical lemma}
is used to bound a function of the form $G_{T}\left(\omega\right)^{d/2}\exp\left(-G{}_{T}^{*}\left(\omega\right)\right)$.
Finally, Lemma \ref{lem:Funny norm} is used to estimate the term
that arises for small $\left\Vert y/T-m\right\Vert $. 
\begin{prop}
\label{lem:bound on I0}For all $\epsilon>0$ , $\tau>0$ and $T\geq\tau^{-1}$,
\[
E_{1}\left(\tau,\epsilon,T\right)\ll\left(\bigl\Vert\widehat{\zeta}_{\tau}\bigr\Vert_{1}+\tau^{\left(1-nd\right)/2}\right)T^{\left(n-1\right)d-n+1}.
\]
\end{prop}
\begin{proof}
Recall (see Section \ref{sec:Set-up.}) 
\[
E_{1}\left(\tau,\epsilon,T\right)=\left|\int_{\mathcal{B}\left(T\right)}R\left(e_{Q,t}w_{\tau,T}\right)\widehat{S}_{\epsilon}\left(\omega\right)d\omega\right|
\]
and 
\[
R\left(e_{Q,\omega}w_{\tau,T}\right)=\int_{\mathbb{R}^{nd}}R\left(e_{Q,\omega}\tilde{e}_{Q,T,y}\right)\widehat{\zeta}_{\tau}\left(y\right)dy=\int_{\mathbb{R}^{nd}}\left(\theta_{T,y}\left(\omega\right)-\vartheta_{T,y}\left(\omega\right)\right)\widehat{\zeta}_{\tau}\left(y\right)dy.
\]
Note that $e_{Q,\omega}\tilde{e}_{Q,T,y}\in\mathcal{S}\left(\mathbb{R}^{nd}\right)$
and thus, there exists a constant $C$, so that $\bigl|e_{Q,\omega}\tilde{e}_{Q,T,y}\left(v\right)\bigr|+\bigl|\widehat{e_{Q,\omega}\tilde{e}_{Q,T,y}}\left(v\right)\bigr|\leq C\left(1+\left\Vert x\right\Vert \right)^{-\left(n+1\right)}$.
Hence, using Poisson summation (Theorem 3.1.17 in \cite{MR2445437})
\begin{equation}
\theta_{T,y}\left(\omega\right)-\vartheta_{T,y}\left(\omega\right)=\sum_{m\in\mathbb{Z}^{nd}\setminus\left\{ 0\right\} }\vartheta_{T,y-Tm}\left(\omega\right).\label{eq:poissoin summation}
\end{equation}
By using (\ref{eq:poissoin summation}), 
\begin{equation}
E_{1}\left(\tau,\epsilon,T\right)=\left|\int_{\mathcal{B}\left(T\right)}\widehat{S}_{\epsilon}\left(\omega\right)\int_{\mathbb{R}^{nd}}\sum_{m\in\mathbb{Z}^{nd}\setminus\left\{ 0\right\} }\vartheta_{T,y-Tm}\left(\omega\right)\widehat{\zeta}_{\tau}\left(y\right)dyd\omega\right|.\label{eq:first equation}
\end{equation}
Let $\Sigma_{T,\omega,y}=\sum_{m\in\mathbb{Z}^{nd}\setminus\left\{ 0\right\} }\vartheta_{T,y-Tm}\left(\omega\right)$.
By Lemma \ref{lem:bound on smooth theta} we have 
\begin{alignat*}{1}
\left|\Sigma_{T,\omega,y}\right| & \ll\sum_{m\in\mathbb{Z}^{nd}\setminus\left\{ 0\right\} }\left|\vartheta_{y-Tm}\left(\omega\right)\right|\ll T^{nd/2}G_{T}\left(\omega\right)^{d/2}\sum_{m\in\mathbb{Z}^{nd}\setminus\left\{ 0\right\} }\exp\left(-Q_{T,\omega}^{*}\left(y,m\right)\right),
\end{alignat*}
where 
\[
Q_{T,\omega}^{*}\left(y,m\right)=\frac{1}{4}\sum_{1\leq i\leq n}g_{T}\left(\omega_{i}-\omega_{i-1}\right)^{2}Q_{+}^{-1}\left(y_{i}/T-m_{i}\right).
\]
Note $Q_{+}^{-1}$ is a positive definite quadratic form and because
$\omega\in\mathcal{B}\left(T\right)$ we have (using that $T>1$)
$1/2<g_{T}\left(\omega_{i}-\omega_{i-1}\right)^{2}$ for $1\le i\leq n$.
Recall, $\lambda_{\textrm{max}}$ is the maximum (in terms of absolute
value) eigenvalue of $Q$. Therefore, $\min_{\lambda\in\textrm{sp}\left(Q_{+}^{-1}\right)}\left|\lambda\right|=\lambda_{\textrm{max}}^{-2}$.
Let $c_{Q}=\frac{1}{16}\lambda_{\textrm{max}}^{-2}$. We get 
\begin{alignat}{1}
Q_{T,\omega}^{*}\left(y,m\right) & \geq2c_{Q}\sum_{1\leq i\leq n}\left\Vert y_{i}/T-m_{i}\right\Vert ^{2}\nonumber \\
 & =2c_{Q}\left\Vert y/T-m\right\Vert ^{2}.\label{eq:lower bound on quad form}
\end{alignat}
Also, $Q_{T,\omega}^{*}\left(y,m\right)\geq\frac{1}{4}G{}_{T}^{*}\left(\omega\right)Q_{++}^{-1}\left(y/T-m\right)$.
If $y/T\in\left[-1/2,1/2\right]^{nd}$ then $\left\Vert y/T-m\right\Vert \geq1/2$
for $m\in\mathbb{Z}^{nd}\setminus\left\{ 0\right\} $. Therefore,
$Q_{T,\omega}^{*}\left(y,m\right)\geq c_{Q}G{}_{T}^{*}\left(\omega\right)$
and 
\begin{alignat}{1}
\exp\left(-Q_{T,\omega}^{*}\left(y,m\right)\right) & =\left(\exp\left(-\frac{1}{2}Q_{T,\omega}^{*}\left(y,m\right)\right)\right)^{2}\leq\exp\left(-\frac{1}{2}Q_{T,\omega}^{*}\left(y,m\right)\right)\exp\left(-c_{Q}G{}_{T}^{*}\left(\omega\right)\right).\label{eq:bound on exp (Qbar)}
\end{alignat}
Thus, by combining Lemma \ref{lem:Bound on exp sum}, (\ref{eq:lower bound on quad form})
and (\ref{eq:bound on exp (Qbar)}), 
\begin{equation}
\sum_{m\in\mathbb{Z}^{nd}\setminus\left\{ 0\right\} }\exp\left(-Q_{T,\omega}^{*}\left(y,m\right)\right)\ll\exp\left(-c_{Q}G{}_{T}^{*}\left(\omega\right)\right).\label{eq: bound on exp sum 1}
\end{equation}
If $y/T\in\left[-1/2,1/2\right]^{nd}+m_{0}$ for some $m_{0}\in\mathbb{Z}^{nd}\setminus\left\{ 0\right\} $
then $\left\Vert y/T-m\right\Vert \geq1/2$ for $m\in\mathbb{Z}^{nd}\setminus\left\{ m_{0}\right\} $.
Therefore, we can repeat the above argument and get that 
\begin{equation}
\sum_{m\in\mathbb{Z}^{nd}\setminus\left\{ 0\right\} }\exp\left(-Q_{T,\omega}^{*}\left(y,m\right)\right)\ll\exp\left(-c_{Q}G{}_{T}^{*}\left(\omega\right)\right)+\exp\left(-Q_{T,\omega}^{*}\left(y',0\right)\right),\label{eq:bound on exp sum 2}
\end{equation}
where $y'/T=y/T-m_{0}\in\left[-1/2,1/2\right]^{nd}$. It follows from
(\ref{eq: bound on exp sum 1}) that 
\begin{equation}
\int_{B_{\infty}\left(T/2\right)}\left|\Sigma_{T,\omega,y}\widehat{\zeta}_{\tau}\left(y\right)\right|dy\ll T^{nd/2}G_{T}\left(\omega\right)^{d/2}\exp\left(-c_{Q}G{}_{T}^{*}\left(\omega\right)\right)\bigl\Vert\widehat{\zeta}_{\tau}\bigr\Vert_{1}\label{eq:big bound 05}
\end{equation}
and from (\ref{eq:bound on exp sum 2}) that 
\begin{multline}
\int_{\mathbb{R}^{nd}\setminus B_{\infty}\left(T/2\right)}\left|\Sigma_{T,\omega,y}\widehat{\zeta}_{\tau}\left(y\right)\right|dy\\
\ll T^{nd/2}G_{T}\left(\omega\right)^{d/2}\left(\exp\left(-c_{Q}G{}_{T}^{*}\left(\omega\right)\right)\bigl\Vert\widehat{\zeta}_{\tau}\bigr\Vert_{1}+\int_{\mathbb{R}^{nd}\setminus B_{\infty}\left(T/2\right)}\exp\left(-\overline{Q}_{T,\omega}\left(y',0\right)\right)\bigl|\widehat{\zeta}_{\tau}\left(y\right)\bigr|dy\right).\label{eq:big bound}
\end{multline}
By Lemma \ref{technical lemma} for $\left\{ s_{1},\dots,s_{n}\right\} =\left\{ \omega_{i}-\omega_{i-1}\right\} _{1\leq i\leq n}$,
with $\omega\in\mathcal{B}\left(T\right)$ there exist positive constants
$c_{1},\dots,c_{n}$, and $b_{2},\dots,b_{n}$ such that 
\[
\exp\left(-c_{Q}G{}_{T}^{*}\left(\omega\right)\right)\leq\exp\left(-\sum_{i=1}^{n}\frac{c_{i}}{T^{b_{i}}}g_{T}\left(s_{i}\right)^{2}\right),
\]
where $b_{1}=0$ and $\sum_{i=2}^{n}b_{i}\leq2\left(n-2\right)$.
Consider the function $x^{d/2}\exp\left(-kx^{2}\right)$, it obtains
its maximum at $x=\sqrt{d/4k}$ and this maximum value is $\left(d/4k\right)^{d/4}\exp\left(-d/4\right)$.
Hence 
\begin{equation}
G_{T}\left(\omega\right)^{d/2}\exp\left(-c_{Q}G{}_{T}^{*}\left(\omega\right)\right)\ll T^{\frac{d}{4}\sum_{i=1}^{n}b_{i}}\leq T^{d\left(n-2\right)/2}.\label{eq:simple bound}
\end{equation}
Also note that 
\[
\exp\left(-Q_{T,\omega}^{*}\left(y',0\right)\right)\leq\exp\left(-\frac{G{}_{T}^{*}\left(\omega\right)}{4}Q_{++}^{-1}\left(y'/T\right)\right)
\]
and thus by using Lemma \ref{technical lemma} again 
\[
G_{T}\left(\omega\right)^{d/2}\exp\left(-Q_{T,\omega}^{*}\left(y',0\right)\right)\ll T^{d\left(n-2\right)/2}\left(Q_{++}^{-1}\left(y'/T\right)\right)^{-nd/4}.
\]
Note that $g_{T}\left(x\right)\leq T$ for all $x\in\mathbb{R}$ and
hence we also have the trivial bound
\[
G_{T}\left(\omega\right)^{d/2}\exp\left(-Q_{T,\omega}^{*}\left(y',0\right)\right)\leq T^{nd/2}.
\]
Hence 
\[
G_{T}\left(\omega\right)^{d/2}\exp\left(-Q_{T,\omega}^{*}\left(y',0\right)\right)\ll\min\left\{ T^{nd/2},T^{d\left(n-2\right)/2}\left(Q_{++}^{-1}\left(y'/T\right)\right)^{-nd/4}\right\} \ll\left(\frac{Q_{++}^{-1}\left(y'/T\right)}{T^{2\left(n-2\right)/n}}+\frac{1}{T^{2}}\right)^{-nd/4}.
\]
For $y\in\mathbb{R}^{nd}$, let $\left\Vert y\right\Vert _{\mathbb{Z}^{nd}}=\min_{z\in\mathbb{Z}^{nd}}\left\Vert z-y\right\Vert $.
By rearranging and using the fact that $Q_{++}^{-1}\left(y'/T\right)\geq\frac{1}{\lambda_{\textrm{max}}^{2}}\left\Vert y/T\right\Vert _{\mathbb{Z}^{nd}}$
we get
\begin{multline}
\int_{\mathbb{R}^{nd}\setminus B_{\infty}\left(T/2\right)}\left(\frac{Q_{++}^{-1}\left(y'/T\right)}{T^{2\left(n-2\right)/n}}+\frac{1}{T^{2}}\right)^{-nd/4}\bigl|\widehat{\zeta}_{\tau}\left(y\right)\bigr|dy\\
\ll\int_{\mathbb{R}^{nd}\setminus B_{\infty}\left(T/2\right)}\left(\frac{1}{\lambda_{\textrm{max}}^{2}T^{2\left(n-2\right)/n}}\left\Vert y/T\right\Vert _{\mathbb{Z}^{nd}}+\frac{1}{T^{2}}\right)^{-nd/4}\bigl|\widehat{\zeta}_{\tau}\left(y\right)\bigr|dy\\
=T^{d\left(n-2\right)/2}\int_{\mathbb{R}^{nd}\setminus B_{\infty}\left(T/2\right)}\frac{\bigl|\widehat{\zeta}_{\tau}\left(y\right)\bigr|}{\left(\lambda_{\max}^{-2}\left\Vert y/T\right\Vert _{\mathbb{Z}^{nd}}+T^{-4/n}\right)^{nd/4}}dy.\label{eq:extra e0}
\end{multline}
 Note that 
\begin{alignat*}{1}
\int_{\mathbb{R}^{nd}\setminus B_{\infty}\left(T/2\right)}\frac{\bigl|\widehat{\zeta}_{\tau}\left(y\right)\bigr|}{\left(\lambda_{\textrm{max}}^{-2}\left\Vert y/T\right\Vert _{\mathbb{Z}^{nd}}+T^{-4/n}\right)^{nd/4}}dy & =\int_{\mathbb{R}^{nd}\setminus B_{\infty}\left(T/2\right)}T^{d}\frac{\bigl|\widehat{\zeta}_{\tau}\left(y\right)\bigr|}{\left(\lambda_{\textrm{max}}^{-2}T^{4/n}\left\Vert y/T\right\Vert _{\mathbb{Z}^{nd}}+1\right)^{nd/4}}dy\\
 & \ll\int_{\mathbb{R}^{nd}\setminus B_{\infty}\left(T/2\right)}T^{d}\bigl|\widehat{\zeta}_{\tau}\left(y\right)\bigr|dy.
\end{alignat*}
Hence, if $T\geq\tau^{-1},$ we can apply Lemma \ref{lem:Funny norm}
and use (\ref{eq:extra e0}) to get 
\[
\int_{\mathbb{R}^{nd}\setminus B_{\infty}\left(T/2\right)}\left(\frac{Q_{++}^{-1}\left(y'/T\right)}{T^{2\left(n-2\right)/n}}+\frac{1}{T^{2}}\right)^{-nd/4}\bigl|\widehat{\zeta}_{\tau}\left(y\right)\bigr|dy\ll T^{d\left(n-2\right)/2}\tau^{\left(1-nd\right)/2}.
\]
Finally by using this, (\ref{eq:first equation}), (\ref{eq:big bound 05}),
(\ref{eq:big bound}) and (\ref{eq:simple bound}) we see that, provided
$T\geq\tau^{-1}$, 
\[
E_{1}\left(\tau,\epsilon,T\right)\ll T^{nd-d}\left(\bigl\Vert\widehat{\zeta}_{\tau}\bigr\Vert_{1}+\tau^{\left(1-nd\right)/2}\right)\int_{\mathcal{B}\left(T\right)}\bigl|\widehat{S}_{\epsilon}\left(\omega\right)\bigr|d\omega\ll T^{nd-d-n+1}\left(\bigl\Vert\widehat{\zeta}_{\tau}\bigr\Vert_{1}+\tau^{\left(1-nd\right)/2}\right),
\]
since Lemma \ref{lem:bound on smoothed indictaor of P} implies that
$\int_{\mathcal{B}\left(T\right)}\bigl|\widehat{S}_{\epsilon}\left(\omega\right)\bigr|d\omega\ll T^{1-n}$. 
\end{proof}

\subsection{\label{sub:Main-term}Bound for $E_{2}$.}

This term contributes the main error term. It is easy to see that
\begin{equation}
E_{2}\left(\tau,\epsilon,T\right)\ll\bigl\Vert\widehat{\zeta}_{\tau}\bigr\Vert_{1}\int_{\mathbb{R}^{n-1}\setminus\mathcal{B}\left(T\right)}\bigl|\widehat{S}_{\epsilon}\left(\omega\right)\bigr|\sup_{y\in\mathbb{R}^{nd}}\left|\theta_{T,y}\left(\omega\right)\right|d\omega.\label{eq:easy bound for I 2}
\end{equation}
For $x\in\mathbb{R}$, let 
\[
\psi\left(T,x\right)=\sum_{\left(m,\bar{m}\right)\in\mathbb{Z}^{d}\times\mathbb{Z}^{d}}\exp\left(-H_{T,x}\left(m,\bar{m}\right)\right),
\]
where 
\[
H_{T,x}\left(m,\bar{m}\right)=T^{2}Q_{+}^{-1}\left(m-\frac{2}{\pi}xQ\bar{m}\right)+T^{-2}Q_{+}\left(\bar{m}\right)
\]
is a positive definite quadratic form on $\mathbb{Z}^{2d}$. From
Lemma 3.3 of \cite{G-M} and (\ref{eq:product formula rough theta})
we have that 
\begin{equation}
\left|\theta_{T,y}\left(\omega\right)\right|\ll T^{nd/2}\prod_{i=1}^{n}\sqrt{\psi\left(T,\omega_{i}-\omega_{i-1}\right)}.\label{eq:lemma 3.3}
\end{equation}
We can now use results and the strategy of \cite{G-M}. Namely, the
quadratic form $H_{T,x}$ is parametrised in terms of the action of
certain geodesic and unipotent elements of $SL_{2}\left(\mathbb{R}\right)$.
Write $v\in\mathbb{R}^{2d}$ as $v=\left(v_{1},\dots,v_{d}\right)$,
where $v_{i}\in\mathbb{R}^{2}$. Consider the action of $SL_{2}\left(\mathbb{R}\right)$
on $\mathbb{R}^{2d}$ given by $gv=\left(gv_{1},\dots,gv_{d}\right)$.
This is the action studied in \cite{G-M}, Section 4. For $T>0$,
let $d_{T}=\left(\begin{smallmatrix}T & 0\\
0 & T^{-1}
\end{smallmatrix}\right)$ and for $x\in\mathbb{R}$, let $u_{x}=\left(\begin{smallmatrix}1 & x\\
0 & 1
\end{smallmatrix}\right)$. For $\left(m,\bar{m}\right)\in\mathbb{Z}^{d}\times\mathbb{Z}^{d}$,
let $m'=Q_{+}^{-1/2}m$ and $\bar{m}'=Q_{+}^{-1/2}Q\bar{m}$. Let
$\Pi\in SL_{2d}\left(\mathbb{Z}\right)$ be the permutation matrix
such that 
\[
\Pi\left(m',\bar{m}'\right)=\left(m'_{1},\bar{m}'_{1},m'_{2},\bar{m}'_{2},\dots,m'_{d},\bar{m}'_{d}\right).
\]
Let $\Lambda_{Q}=\Pi\left(\begin{smallmatrix}Q_{+}^{-1/2} & 0\\
0 & Q_{+}^{-1/2}Q
\end{smallmatrix}\right)\mathbb{Z}^{2d}$. It is shown in \cite{G-M} (equation (4.21)) and is indeed not hard
to see, that 
\begin{equation}
H_{T,x}\left(m,\bar{m}\right)=\left\Vert d_{T}u_{\frac{2}{\pi}x}\Pi\left(m',\bar{m}'\right)\right\Vert ^{2}.\label{eq:quadrtaic form h}
\end{equation}
The following Lemma follows from Lemma 3.4 of \cite{G-M}.
\begin{lem}
\label{lem:3.4}Let $\Delta$ be a lattice in $\mathbb{R}^{d}$. Then
\[
\sum_{v\in\Delta}\exp\left(-\left\Vert v\right\Vert ^{2}\right)\ll\left|\left\{ v\in\Delta:\left\Vert v\right\Vert _{\infty}<1\right\} \right|,
\]
where the implicit constant does not depend on $\Delta$. 
\end{lem}
It is easy to see that by using (\ref{eq:quadrtaic form h}) and the
definition of $\psi$ that Lemma \ref{lem:3.4} implies 
\begin{equation}
\psi\left(T,x\right)\ll\left|\left\{ v\in d_{T}u_{\frac{2}{\pi}x}\Lambda_{Q}:\left\Vert v\right\Vert _{\infty}<1\right\} \right|.\label{eq:start of proof for boxes}
\end{equation}
Next we introduce the function $\alpha$ on the space of lattices.
For more details see Section 4 of \cite{G-M}. Let $\Delta\in\mathbb{R}^{2d}$
be a lattice. Let $U$ be a subspace of $\mathbb{R}^{2d}$, we say
that $U$ is $\Delta$-rational if $U\cap\Delta$ is a lattice in
$U$. For $1\leq i\leq2d$ and a quasinorm, $\left|.\right|_{i}$
on $\bigwedge^{i}\left(\mathbb{R}^{2d}\right)$ we define $d_{\Delta}\left(U\right)=\left|u_{1}\wedge\dots\wedge u_{i}\right|_{i}$
where $u_{1},\dots,u_{i}$ is a basis of $U\cap\Delta$ over $\mathbb{Z}$.
Note that $d_{\Delta}\left(U\right)$ does not depend on the choice
of basis. Note that any two quasinorms on $\mathbb{R}^{2d}$ are equivalent.
For details about the precise choice of quasinorm, see section 5 of
\cite{G-M}. Let 
\[
\Psi_{i}\left(\Delta\right)=\left\{ U:U\textrm{ is a }\Delta\textrm{-rational subspace of }\mathbb{R}^{2d}\textrm{ with }\dim U=i\right\} .
\]
Define 
\[
\alpha_{i}\left(\Delta\right)=\sup_{U\in\Psi_{i}\left(\Delta\right)}\frac{1}{d_{\Delta}\left(U\right)}\quad\textrm{and}\quad\alpha\left(\Delta\right)=\max_{1\leq i\leq2d}\alpha_{i}\left(\Delta\right).
\]
The following Lemma collects together several results from \cite{G-M}
regarding the alpha functions.
\begin{lem}
\label{Crucial lemma from G-M}Let $\Lambda_{T,x}=d_{T}u_{x}\Lambda_{Q}$
then 
\begin{enumerate}
\item \label{enu:Lem 4.6}For any $x\in\mathbb{R}$ and $\mu>1$, $\left|\left\{ v\in\Lambda_{T,x}:\left\Vert v\right\Vert \leq\mu\right\} \right|\asymp\mu^{2d}\alpha_{d}\left(\Lambda_{T,x}\right)$. 
\item \label{enu:Cor 4.7}For any $x\in\mathbb{R}$, \foreignlanguage{english}{\textup{$\alpha\left(\Lambda_{T,x}\right)\asymp\alpha_{d}\left(\Lambda_{T,x}\right)$}}.
\item \label{enu:Lem 4.8 i}If $T\geq\lambda_{\textrm{\ensuremath{\max}}}$,
then $\sup_{x\in\mathbb{R}}\alpha_{d}\left(\Lambda_{T,x}\right)\ll T^{d}$. 
\item \label{enu:Lem 4.8 ii}If $\left|x\right|\geq\left(\lambda_{\min}T\right)^{-1}$,
then $\alpha_{d}\left(\Lambda_{T,x}\right)\ll\left(T^{-1}+\left|xT\right|\right)^{d}$.
\item \label{enu:Cor 4.11 i}Let $I=\left[a,b\right]$ with $a\in\left(0,1\right)$
and $b>2$. If $Q$ is irrational then $\lim_{T\rightarrow\infty}\left(\sup_{x\in I}\alpha_{d}\left(\Lambda_{T,x}\right)T^{-d}\right)=0$. 
\item \label{enu:Cor 4.11 ii}Let $I=\left[a,b\right]$ with $a\in\left(0,1\right)$
and $b>2$. If $Q$ is of Diophantine type $\left(\kappa,A\right)$
then $\sup_{x\in I}\alpha_{d}\left(\Lambda_{T,x}\right)T^{-d}\ll\max\left\{ a^{-\left(\kappa+1\right)},b^{\kappa}\right\} T^{2\left(\kappa-1\right)}$. 
\end{enumerate}
\end{lem}
These facts are all proved in Section 4 of \cite{G-M}. For (\ref{enu:Lem 4.6})
see Lemma 4.6. For (\ref{enu:Cor 4.7}) see corollary 4.7. For (\ref{enu:Lem 4.8 i})
and (\ref{enu:Lem 4.8 ii}) see Lemma 4.8. For (\ref{enu:Cor 4.11 i})
and (\ref{enu:Cor 4.11 ii}) see Corollary 4.11. 

We proceed by finding an approximation of the integral in (\ref{eq:easy bound for I 2})
over a fixed box in $\mathbb{R}^{n-1}$. As in \cite{G-M} we use
(\ref{eq:lemma 3.3}) together with the parametrisation of the quadratic
form $H_{T,x}$ discussed previously. The function $\alpha$ is then
introduced via (\ref{eq:start of proof for boxes}) and Lemma \ref{Crucial lemma from G-M},
part (\ref{enu:Lem 4.6}). Let $K=SO_{2}\left(\mathbb{R}\right)$.
A change of variables is used to convert the problem from an integral
over a box in $\mathbb{R}^{n-1}$ into an integral over $K^{n-1}$.
This is done via Lemma 4.9 of \cite{G-M}, which is reproduced below.
\begin{lem}
\label{lem:change of variables}Let $x\in\left[-2,2\right]$, $T>1$
and $\Delta$ be a lattice in $\mathbb{R}^{2d}$, then for all $1\leq i\leq2d$,
\[
\alpha_{i}\left(d_{T}u_{x}\Delta\right)\ll\alpha_{i}\left(d_{T}k_{\theta}\Delta\right),
\]
where $k_{\theta}=\left(\begin{smallmatrix}\cos\theta & -\sin\theta\\
\sin\theta & \cos\theta
\end{smallmatrix}\right)\in SO_{2}\left(\mathbb{R}\right)$ and $x=\tan\theta$. 
\end{lem}
Then the problem is to understand the averages over translated $K$
orbits. This problem is studied in Section 5 of \cite{G-M}. In particular,
the following Theorem (Theorem 5.11 in \cite{G-M}) is the critical
estimate. 
\begin{thm}
\label{thm:bound on average over k - orbit}Let $2/d<\beta$. Then
there exists a constant $C$ depending only on $\beta$ and the choice
of quasinorms used to define $\alpha$, such that for all $g\in SL_{2}\left(\mathbb{R}\right)$
and any lattice $\Delta\in\mathbb{R}^{2d}$,
\[
\int_{K}\alpha\left(gk\Delta\right)^{\beta}d\nu\left(k\right)\leq C\alpha\left(\Delta\right)^{\beta}\left\Vert g\right\Vert ^{\beta d-2}.
\]

\end{thm}
As previously remarked the bounds for $E_{2}$ contain the arithmetic
information regarding the quadratic form $Q$. This is encoded via
the following function. For an interval $I$, and $2/d<\beta<1/2$
define $\gamma_{I,\beta}\left(T\right)=\sup_{x\in I}\left(T^{-d}\alpha_{d}\left(\Lambda_{T,x}\right)\right)^{1/2-\beta}$.
Let $L_{0}=L_{n}$ be the empty set, $L_{1},\dots,L_{n-1}$ be intervals
and $\Gamma_{L_{1},\dots,L_{n-1},\beta}\left(T\right)=\prod_{i=1}^{n}\gamma_{L_{i}-L_{i-1},\beta}\left(T\right)$
where the notation $L_{i}-L_{j}=\left\{ l_{i}-l_{j}:l_{i}\in L_{i},l_{j}\in L_{j}\right\} $.
\begin{lem}
\label{lem: bound on hard piece fixed intreval}Let $L_{1},\dots,L_{n-1}$
be closed intervals of length $2$, $L=L_{1}\times\dots\times L_{n-1}$
and $S_{\epsilon}^{*}\left(L\right)=\sup_{\omega\in L}\bigl|\widehat{S}_{\epsilon}\left(\omega\right)\bigr|$.
Then for all $\epsilon>0$ and $T\geq\lambda_{\textrm{\ensuremath{\max}}}$,
\[
\int_{L}\bigl|\widehat{S}_{\epsilon}\left(\omega\right)\bigr|\sup_{y\in\mathbb{R}^{nd}}\left|\theta_{T,y}\left(\omega\right)\right|d\omega\ll S_{\epsilon}^{*}\left(L\right)\Gamma_{L_{1},\dots,L_{n-1},\beta}\left(T\right)T^{nd-2\left(n-1\right)}.
\]
\end{lem}
\begin{proof}
By using (\ref{eq:start of proof for boxes}) and Lemma \ref{Crucial lemma from G-M},
part (\ref{enu:Lem 4.6}) we get 
\[
\psi\left(T,x\right)\ll\left|\left\{ v\in\Lambda_{T,\frac{2}{\pi}x}:\left\Vert v\right\Vert \leq d^{1/2}\right\} \right|\ll\alpha_{d}\left(\Lambda_{T,\frac{2}{\pi}x}\right).
\]
It easily follows from the previous formula and the definition of
$\gamma_{I,\beta}\left(T\right)$ that for all $x\in\frac{\pi}{2}I$,
\begin{equation}
\psi\left(T,x\right)^{1/2}\ll T^{d/2-\beta d}\gamma_{I,\beta}\left(T\right)\alpha_{d}\left(\Lambda_{T,\frac{2}{\pi}x}\right)^{\beta}.\label{eq:bound on psi}
\end{equation}
Note that by Lemma \ref{Crucial lemma from G-M}, part (\ref{enu:Lem 4.8 i}),
$\sup_{\omega_{n-1}\in\mathbb{R}}\left(\alpha_{d}\left(\Lambda_{T,-\frac{2}{\pi}\omega_{n-1}}\right)^{\beta}\right)\ll T^{d\beta}.$
Therefore, by using (\ref{eq:lemma 3.3}) and (\ref{eq:bound on psi}),
for $\omega\in\frac{\pi}{2}L$ we have
\begin{alignat*}{1}
\left|\theta_{T,y}\left(\omega\right)\right| & \ll T^{nd-n\beta d}\Gamma_{L_{1},\dots,L_{n-1},\beta}\left(T\right)\prod_{i=1}^{n}\alpha_{d}\left(\Lambda_{T,\frac{2}{\pi}\left(\omega_{i}-\omega_{i-1}\right)}\right)^{\beta}\\
 & \ll T^{nd-\left(n-1\right)\beta d}\Gamma_{L_{1},\dots,L_{n-1},\beta}\left(T\right)\prod_{i=1}^{n-1}\alpha_{d}\left(\Lambda_{T,\frac{2}{\pi}\left(\omega_{i}-\omega_{i-1}\right)}\right)^{\beta}.
\end{alignat*}
Note that $L\subset\frac{\pi}{2}L$, hence 
\begin{equation}
\int_{L}\bigl|\widehat{S}_{\epsilon}\left(\omega\right)\bigr|\sup_{y\in\mathbb{R}^{nd}}\left|\theta_{T,y}\left(\omega\right)\right|d\omega\ll S_{\epsilon}^{*}\left(L\right)\Gamma_{L_{1},\dots,L_{n-1},\beta}\left(T\right)T^{nd-\left(n-1\right)\beta d}\prod_{i=1}^{n-1}\int_{L_{i}-L_{i-1}}\alpha_{d}\left(\Lambda_{T,\xi_{i}}\right)^{\beta}d\xi_{i},\label{eq:bound fixed intervals half way}
\end{equation}
where we did the change of variables $\frac{2}{\pi}\left(\omega_{i}-\omega_{i-1}\right)\rightarrow\xi_{i}$
for $1\leq i\leq d-1$. (Note, $\frac{2}{\pi}\left(L_{i}-L_{i-1}\right)\subset L_{i}-L_{i-1}$.)
Let $l_{i,0}$ denote the midpoints of the intervals $L_{i}-L_{i-1}$.
Note that $L_{i}-L_{i-1}$ has length $4$. Partition the intervals
$L_{i}-L_{i-1}$ via $l_{i,j}=l_{i,0}-2+j\frac{2}{\lambda_{\textrm{\ensuremath{\max}}}}$
for $j\in\left[0,2\lambda_{\textrm{\ensuremath{\max}}}+1\right]\cap\mathbb{Z}$.
Note that $d_{T}u_{\xi_{i}}=d_{T/\lambda_{\textrm{\ensuremath{\max}}}}u_{s_{i,j}}d_{\lambda_{\textrm{\ensuremath{\max}}}}u_{l_{i,j}}$,
where $s_{i,j}=\left(\xi_{i}-l_{i,j}\right)\lambda_{\textrm{\ensuremath{\max}}}^{2}$.
Changing variables from $\xi_{i}$ to $s_{i,j}$ we get,
\begin{alignat}{1}
\int_{L_{i}-L_{i-1}}\alpha_{d}\left(\Lambda_{T,\xi_{i}}\right)^{\beta}d\xi_{i} & \ll\sum_{j\in\left[0,2\lambda_{\textrm{\ensuremath{\max}}}+1\right]\cap\mathbb{Z}}\int_{\left[l_{i,j-1},l_{i,j}\right]}\alpha_{d}\left(d_{T/\lambda_{\textrm{\ensuremath{\max}}}}u_{s_{i,j}}d_{\lambda_{\textrm{\ensuremath{\max}}}}u_{l_{i,j}}\Lambda_{Q}\right)^{\beta}ds_{i,j}\nonumber \\
 & \ll\max_{j\in\left[0,2\lambda_{\textrm{\ensuremath{\max}}}+1\right]\cap\mathbb{Z}}\int_{-2}^{2}\alpha_{d}\left(d_{T/\lambda_{\textrm{\ensuremath{\max}}}}u_{s_{i,j}}\Lambda_{Q,l_{i,j}}\right)^{\beta}ds_{i,j},\label{eq:fixed interavls 2}
\end{alignat}
where $\Lambda_{Q,l_{i,j}}=d_{\lambda_{\textrm{\ensuremath{\max}}}}u_{l_{i,j}}\Lambda_{Q}$.
By using Lemma \ref{lem:change of variables} we see that 
\begin{equation}
\int_{-2}^{2}\alpha_{d}\left(d_{T/\lambda_{\textrm{\ensuremath{\max}}}}u_{s_{i,j}}\Lambda_{Q,l_{i,j}}\right)^{\beta}ds_{i,j}\ll\int_{-\pi}^{\pi}\alpha_{d}\left(d_{T/\lambda_{\textrm{\ensuremath{\max}}}}k_{\theta_{i,j}}\Lambda_{Q,l_{i,j}}\right)^{\beta}d\theta_{i,j},\label{eq:fixed intervals 3}
\end{equation}
where $\tan\theta_{i,j}=s_{i,j}$. By Theorem \ref{thm:bound on average over k - orbit},
\begin{equation}
\int_{-\pi}^{\pi}\alpha_{d}\left(d_{T/\lambda_{\textrm{\ensuremath{\max}}}}k_{\theta_{i,j}}\Lambda_{Q,l_{i,j}}\right)^{\beta}d\theta_{i}\ll\alpha\left(\Lambda_{Q,l_{i,j}}\right)T^{\beta d-2}.\label{eq:fixed intervals 4}
\end{equation}
Note that Lemma \ref{Crucial lemma from G-M}, part (\ref{enu:Cor 4.7})
implies $\alpha\left(\Lambda_{Q,l_{i,j}}\right)\ll\alpha_{d}\left(\Lambda_{Q,l_{i,j}}\right)$
and Lemma \ref{Crucial lemma from G-M}, part (\ref{enu:Lem 4.8 i})
implies 
\[
\max_{j\in\left[0,2\lambda_{\textrm{\ensuremath{\max}}}+1\right]\cap\mathbb{Z}}\alpha_{d}\left(\Lambda_{Q,l_{i,j}}\right)\ll1.
\]
Hence, using (\ref{eq:bound fixed intervals half way}), (\ref{eq:fixed interavls 2}),
(\ref{eq:fixed intervals 3}) and (\ref{eq:fixed intervals 4}) we
get the conclusion of the Lemma. 
\end{proof}
We can now prove the bound for $E_{2}$. Recall $T_{n}=\left(n-1\right)T$.
Let 
\begin{equation}
\mathscr{A}_{\epsilon}\left(T\right)=\underset{N_{-}\in\left(T_{n}^{-1},1\right)}{\inf_{N\in\left(1,\infty\right)}}\left\{ \log\left(1/\epsilon\right)^{n-1}\left(N_{-}^{d\left(1/2-\beta\right)}+\gamma_{\left[N_{-},N\right],\beta}\left(T\right)\right)+\exp\left(-\left(n-1\right)c\sqrt{\epsilon N}\right)\right\} .\label{eq:def of A_epT}
\end{equation}
By Lemma \ref{Crucial lemma from G-M}, part (\ref{enu:Cor 4.11 i})
for any fixed $\epsilon>0$ we have $\lim_{T\rightarrow\infty}\mathscr{A}_{\epsilon}\left(T\right)=0$
provided that $Q$ is irrational. 
\begin{prop}
\label{lem:Bound on I2}For all $0<\epsilon<1$ , $\tau>0$ and $T\geq\lambda_{\max}$,
\[
E_{2}\left(\tau,\epsilon,T\right)\ll\bigl\Vert\widehat{\zeta}_{\tau}\bigr\Vert_{1}\mathscr{A}_{\epsilon}\left(T\right)T^{nd-2\left(n-1\right)}.
\]
\end{prop}
\begin{proof}
Let $J_{-1}=\left[0,\left(2T\right)^{-1}\right]$, $J_{0}=\left[\left(2T\right)^{-1},1\right]$
and $J_{i}=\left[i,i+1\right]$ for $i\geq1$. We consider the only
the portion of $\mathbb{R}^{n-1}\setminus\mathcal{B}\left(T\right)$
lying in the positive cone since bounds for the other cones can be
obtained in an identical manner. Note that $\mathbb{R}^{n-1}\setminus\mathcal{B}\left(T\right)\subset\mathbb{R}^{n-1}\setminus\left[-\left(2T\right)^{-1},\left(2T\right)^{-1}\right]^{n-1}$,
therefore the portion of $\mathbb{R}^{n-1}\setminus\mathcal{B}\left(T\right)$
lying in the positive cone is contained in 
\[
\underset{\left(i_{1}\dots,i_{n-1}\right)\neq\left(-1,\dots,-1\right)}{\bigcup_{-1\leq i_{1},\dots,i_{n-1}}}J_{i_{1}}\times\dots\times J_{i_{n-1}}.
\]
By Lemma \ref{lem: bound on hard piece fixed intreval}, 
\begin{equation}
\int_{J_{i_{1}}\times\dots\times J_{i_{n-1}}}\bigl|\widehat{S}_{\epsilon}\left(\omega\right)\bigr|\sup_{y\in\mathbb{R}^{nd}}\left|\theta_{T,y}\left(\omega\right)\right|d\omega\leq S_{\epsilon}^{*}\left(J_{i_{1}}\times\dots\times J_{i_{n-1}}\right)\Gamma_{J_{i_{1}},\dots,J_{i_{n-1}},\beta}\left(T\right)T^{nd-2\left(n-1\right)}.\label{eq:estimate from lemma 7}
\end{equation}
From Lemma \ref{lem:bound on smoothed indictaor of P} we have 
\begin{equation}
S_{\epsilon}^{*}\left(J_{i_{1}}\times\dots\times J_{i_{n-1}}\right)\ll\prod_{j=1}^{n-1}\min\left\{ 1,\left|1/i_{j}\right|\right\} \exp\left(-c\sqrt{\epsilon i_{j}}\right).\label{eq:end proof 3}
\end{equation}
For $N>1$, 
\begin{alignat}{1}
\sum_{1\leq i\leq N}\frac{1}{i}\exp\left(-c\sqrt{\epsilon i}\right)\leq\int_{1}^{N}\frac{1}{x}\exp\left(-c\sqrt{\epsilon x}\right)dx & =\int_{c\sqrt{\epsilon}}^{c\sqrt{\epsilon N}}\frac{1}{x}\exp\left(-x\right)dx\nonumber \\
 & \leq\int_{c\sqrt{\epsilon}}^{1}\frac{1}{x}\exp\left(-x\right)dx+\int_{1}^{\infty}\frac{1}{x}\exp\left(-x\right)dx\nonumber \\
 & \ll\log\left(1/\epsilon\right).\label{eq:calc}
\end{alignat}
Hence, (\ref{eq:end proof 3}) gives 
\begin{equation}
\underset{\left(i_{1}\dots,i_{n-1}\right)\neq\left(-1,\dots,-1\right)}{\sum_{-1\leq i_{1},\dots,i_{n-1}\leq N}}S_{\epsilon}^{*}\left(J_{i_{1}}\times\dots\times J_{i_{n-1}}\right)\ll\log\left(1/\epsilon\right)^{n-1}.\label{eq:bounds on sups of weights}
\end{equation}
Note that $\left\{ \omega\in\mathbb{R}^{n-1}:\left|\omega_{i}-\omega_{i-1}\right|\leq T_{n}^{-1}\right\} \subset\left[-\left(2T\right)^{-1},\left(2T\right)^{-1}\right]^{n-1}$
and therefore, for all $-1\leq i_{1},\dots,i_{n-1}\leq N$ with $\left(i_{1}\dots,i_{n-1}\right)\neq\left(-1,\dots,-1\right)$
at least one of the $J_{i_{j}}-J_{i_{j-1}}$ is contained in $\left[T_{n}^{-1},N\right]$.
From Lemma \textcolor{black}{\ref{Crucial lemma from G-M}, part (\ref{enu:Lem 4.8 i})
}it follows that $\gamma_{I,\beta}\left(T\right)\ll1$ for any interval
$I$, moreover, if $I\subset I'$ then $\gamma_{I,\beta}\left(T\right)\leq\gamma_{I',\beta}\left(T\right)$.
Thus, $\Gamma_{J_{i_{1}},\dots,J_{i_{n-1}},\beta}\left(T\right)\ll\gamma_{\left[T_{n}^{-1},N\right],\beta}\left(T\right)$.
Hence using (\ref{eq:bounds on sups of weights}),
\begin{equation}
\underset{\left(i_{1}\dots,i_{n-1}\right)\neq\left(-1,\dots,-1\right)}{\sum_{-1\leq i_{1},\dots,i_{n-1}\leq N}}S_{\epsilon}^{*}\left(J_{i_{1}}\times\dots\times J_{i_{n-1}}\right)\Gamma_{J_{i_{1}},\dots,J_{i_{n-1}},\beta}\left(T\right)\ll{\displaystyle \log\left(1/\epsilon\right)^{n-1}\gamma_{\left[T_{n}^{-1},N\right],\beta}\left(T\right)}.\label{eq:end of proof 4}
\end{equation}
Moreover, Lemma \ref{Crucial lemma from G-M}, part (\ref{enu:Lem 4.8 i})
and a similar calculation as (\ref{eq:calc}) yields 
\begin{equation}
\sum_{N<i_{1},\dots,i_{n-1}}S_{\epsilon}^{*}\left(J_{i_{1}}\times\dots\times J_{i_{n-1}}\right)\Gamma_{J_{i_{1}},\dots,J_{i_{n-1}},\beta}\left(T\right)\ll\exp\left(-\left(n-1\right)c\sqrt{\epsilon N}\right).\label{eq:estimate for the tail}
\end{equation}
Split the interval $\left[T_{n}^{-1},N\right]=\left[T_{n}^{-1},N_{-}\right]\cup\left[N_{-},N\right]$
where $N_{-}\in\left(T_{n}^{-1},1\right)$. We see that 
\[
\gamma_{\left[T_{n}^{-1},N\right],\beta}\left(T\right)\leq\gamma_{\left[T_{n}^{-1},N_{-}\right],\beta}\left(T\right)+\gamma_{\left[N_{-},N\right],\beta}\left(T\right).
\]
\textcolor{black}{By Lemma \ref{Crucial lemma from G-M}, part (\ref{enu:Lem 4.8 ii})
provided that $\lambda_{\min}\geq n-1$, we get 
\begin{equation}
\gamma_{\left[T_{n}^{-1},N_{-}\right],\beta}\left(T\right)\ll N_{-}^{d\left(1/2-\beta\right)}.\label{eq:estimate on small t}
\end{equation}
Using (\ref{eq:easy bound for I 2}) combined with the estimates (\ref{eq:estimate from lemma 7}),
(\ref{eq:end of proof 4}), (\ref{eq:estimate for the tail}) and
(\ref{eq:estimate on small t}) }we get that for all $\tau>0$, $0<\epsilon<1$,
$T>\lambda_{\textrm{\ensuremath{\max}}}$, $N_{-}\in\left(T_{n}^{-1},1\right)$
and $N>1$, 
\[
E_{2}\left(\tau,\epsilon,T\right)\ll\bigl\Vert\widehat{\zeta}_{\tau}\bigr\Vert_{1}\left(\log\left(1/\epsilon\right)^{n-1}\left(N_{-}^{d\left(1/2-\beta\right)}+\gamma_{\left[N_{-},N\right],\beta}\left(T\right)\right)+\exp\left(-\left(n-1\right)c\sqrt{\epsilon N}\right)\right)T^{nd-2\left(n-1\right)},
\]
the claim of the Lemma follows. 
\end{proof}

\section{\label{sec:Proof-of-main}Proof of the main Theorems}

In this section we combine the results from the previous sections
to prove Theorems \ref{thm:Pair correlation} and \ref{thm:Pair correlation-1}.
Throughout this section the assertions will hold with the parameter
$T$ larger than some constant, which will be called $T_{0}$. However
the actual value of $T_{0}$ may change from one occurrence to the
next. In principle the the actual values of $T_{0}$ can be determined
by analysing the proofs, but we will not do this here. 
\begin{lem}
\label{lem:End}For all $\tau\in\left(0,1/2\right)$ and $\epsilon\in\left(0,1\right)$,
there exists $T_{0}>0$ such that for all $T>\max\left(T_{0},\tau^{-1}\right)$,
\[
\left|\frac{\left|\mathbb{Z}^{nd}\cap P_{Q}^{n}\left(I_{1},\dots,I_{n-1}\right)\cap B\left(T\right)\right|}{\mathrm{Vol}\left(P_{Q}^{n}\left(I_{1},\dots,I_{n-1}\right)\cap B\left(T\right)\right)}-1\right|\ll\mathscr{E}_{\epsilon,\tau}\left(T\right),
\]
where $\mathscr{E}_{\epsilon,\tau}\left(T\right)=\left(1/\tau\right)^{\left(nd-1\right)/2}\left(T^{n-1-d}+\mathscr{A}_{\epsilon}\left(T\right)\right)+\epsilon+\tau$. \end{lem}
\begin{proof}
First note that from Corollary \ref{cor:volume}, there exists $T_{0}>0$
such that for all $T>T_{0}$, $\mathrm{Vol}\left(P_{Q}^{n}\left(I_{1},\dots,I_{n-1}\right)\cap B\left(T\right)\right)\gg T^{nd-2\left(n-1\right)}$
. Therefore
\begin{equation}
\left|\frac{\left|\mathbb{Z}^{nd}\cap P_{Q}^{n}\left(I_{1},\dots,I_{n-1}\right)\cap B\left(T\right)\right|}{\mathrm{Vol}\left(P_{Q}^{n}\left(I_{1},\dots,I_{n-1}\right)\cap B\left(T\right)\right)}-1\right|\ll\frac{\left(\left|\mathbb{Z}^{nd}\cap P_{Q}^{n}\left(I_{1},\dots,I_{n-1}\right)\cap B\left(T\right)\right|-\mathrm{Vol}\left(P_{Q}^{n}\left(I_{1},\dots,I_{n-1}\right)\cap B\left(T\right)\right)\right)}{T^{nd-2\left(n-1\right)}}.\label{eq:end 1}
\end{equation}
From Lemmas \ref{Cor::smooth approximation 1} , \ref{cor:smooth approx 2}
and Corollaries \ref{cor:tau} and \ref{cor:epsilon} for all $T>T_{0}$,
$\tau\in\left(0,1/2\right)$ and $\epsilon\in\left(0,1\right)$ we
get
\begin{equation}
\frac{\left(\left|\mathbb{Z}^{nd}\cap P_{Q}^{n}\left(I_{1},\dots,I_{n-1}\right)\cap B\left(T\right)\right|-\mathrm{Vol}\left(P_{Q}^{n}\left(I_{1},\dots,I_{n-1}\right)\cap B\left(T\right)\right)\right)}{T^{nd-2\left(n-1\right)}}\ll\frac{\max_{\pm\epsilon,\pm\tau}\left|R\left(S_{\pm\epsilon}^{Q}w_{\pm\tau,T}\right)\right|}{T^{nd-2\left(n-1\right)}}+\epsilon+\tau.\label{eq:end 2}
\end{equation}
From (\ref{eq:decomposition}) and Propositions \ref{lem:Bound on I1-1},
\ref{lem:bound on I0} and \ref{lem:Bound on I2}, we get that that
for all $\tau\in\left(0,1/2\right)$, $\epsilon\in\left(0,1\right)$
and $T\geq\max\left(\tau^{-1},T_{0}\right)$. 
\begin{alignat}{1}
\frac{\max_{\pm\epsilon,\pm\tau}\left|R\left(S_{\pm\epsilon}^{Q}w_{\pm\tau,T}\right)\right|}{T^{nd-2\left(n-1\right)}} & \ll\bigl\Vert\widehat{\zeta}_{\tau}\bigr\Vert_{1}T^{2-d}+T^{n-1-d}\left(\bigl\Vert\widehat{\zeta}_{\tau}\bigr\Vert_{1}+\left(1/\tau\right)^{\left(nd-1\right)/2}\right)+\bigl\Vert\widehat{\zeta}_{\tau}\bigr\Vert_{1}\mathscr{A}_{\epsilon}\left(T\right).\label{eq:end 3}
\end{alignat}
Note that $2-d\leq n-1-d$ for $n\geq3$. Finally, using Corollary
\ref{cor:l1 norm of zeta} to bound $\bigl\Vert\widehat{\zeta}_{\tau}\bigr\Vert_{1}\ll\left(1/\tau\right)^{\left(nd-1\right)/2}$
we get the conclusion of the Lemma from (\ref{eq:end 1}), (\ref{eq:end 2})
and (\ref{eq:end 3}). 
\end{proof}
The proof of Theorem \ref{thm:Pair correlation} follows immediately
from Lemma \ref{lem:End}. 
\begin{proof}[Proof of Theorem \ref{thm:Pair correlation}]
Note that for $n\leq d$ and all $\tau\in\left(0,1/2\right)$ and
$\epsilon\in\left(0,1\right)$, 
\[
\lim_{T\rightarrow\infty}\left(\left(1/\tau\right)^{\left(nd-1\right)/2}\left(T^{n-1-d}+\mathscr{A}_{\epsilon}\left(T\right)\right)+\epsilon+\tau\right)=\left(1/\tau\right)^{\left(nd-1\right)/2}\lim_{T\rightarrow\infty}\mathscr{A}_{\epsilon}\left(T\right)+\epsilon+\tau.
\]
By Lemma \ref{Crucial lemma from G-M}, part (\ref{enu:Cor 4.11 i})
for all $\epsilon>0$, $\lim_{T\rightarrow\infty}\mathscr{A}_{\epsilon}\left(T\right)=0$,
when $Q$ is irrational. Hence, for all $\tau\in\left(0,1/2\right)$
and $\epsilon\in\left(0,1\right)$, $\lim_{T\rightarrow\infty}\mathscr{E}_{\epsilon,\tau}\left(T\right)=\epsilon+\tau$
and this is the claim of Theorem \ref{thm:Pair correlation}. 
\end{proof}
Note that in order to prove Theorem \ref{thm:Pair correlation} we
did not need to use bounds for $\bigl\Vert\widehat{\zeta}_{\tau}\bigr\Vert_{1}$
from subsection \ref{sub:Norms}. It was only necessary to know that
for any $\tau>0$ the quantities involving $\tau$ remain finite.
In order to prove Theorem \ref{thm:Pair correlation-1} we will use
the bounds from subection \ref{sub:Norms} together with bounds from
Section \ref{sec:Bounding-the-integrals} and results from \cite{G-M}
(see Lemma \ref{Crucial lemma from G-M}, part (\ref{enu:Cor 4.11 ii}))
that give an explicit rate for the convergence of $\lim_{T\rightarrow\infty}\gamma_{\left[N_{-},N_{+}\right],\beta}\left(T\right)$
when $Q$ is of Diophantine type $\left(\kappa,A\right)$. 
\begin{proof}[Proof of Theorem \ref{thm:Pair correlation-1}]
Recall the definition of $\mathscr{A}_{\epsilon}\left(T\right)$.
Take $N_{+}=\frac{\log^{2}\left(1/\epsilon\right)}{\left(n-1\right)^{2}\epsilon c^{2}}$.
Then $\exp\left(-\left(n-1\right)c\sqrt{\epsilon N_{+}}\right)=\epsilon$.
Then for all $\tau\in\left(0,1/2\right)$, $\epsilon\in\left(0,1\right)$
and $T\geq1$ 
\[
\mathscr{E}_{\epsilon,\tau}\left(T\right)\leq\inf_{N_{-}\in\left(T_{n}^{-1},1\right)}\left\{ \left(1/\tau\right)^{\left(nd-1\right)/2}\log^{n-1}\left(1/\epsilon\right)\left(N_{-}^{d\left(1/2-\beta\right)}+\gamma_{\left[N_{-},N_{+}\right],\beta}\left(T\right)+T^{n-1-d}\right)+\epsilon+\tau\right\} .
\]
Let $\sigma>0$, set $\tau=\epsilon=N_{-}^{\sigma}$, $\left(nd-1\right)/2=q$,
$1/2-\beta=\beta'$, then for all $T\geq1$, 
\[
\mathscr{E}_{\epsilon,\tau}\left(T\right)\leq\inf_{N_{-}\in\left(T_{n}^{-1},1\right)}\left\{ \left|\log\left(N_{-}\right)\right|^{n-1}\left(N_{-}^{\beta'd-\sigma q}+N_{-}^{-\sigma q}\gamma_{\left[N_{-},N_{+}\right],\beta}\left(T\right)+N_{-}^{\sigma}+N_{-}^{-\sigma q}T^{n-1-d}\right)\right\} .
\]
By Lemma \ref{Crucial lemma from G-M}, part (\ref{enu:Cor 4.11 ii}),
if $Q$ is of type $\left(\kappa,A\right)$, then $\gamma_{\left[N_{-},N_{+}\right],\beta}\left(T\right)\ll\left(\max\left\{ N_{-}^{-1-\kappa},N_{+}^{\kappa}\right\} T^{-2\left(1-\kappa\right)}\right)^{\beta'}$for
all $T\geq1$. Note that 
\[
\max\left\{ N_{-}^{-1-\kappa},N_{+}^{\kappa}\right\} \ll\max\left\{ \frac{1}{N_{-}^{1+\kappa}},\frac{\left|\log\left(N_{-}\right)\right|^{2\kappa}}{N_{-}^{\kappa\sigma}}\right\} \ll\frac{1}{N_{-}^{1+\kappa}},
\]
provided that $\sigma<1$. Let $2\left(1-\kappa\right)=\kappa'$,
then for all $T\geq1$, 
\begin{equation}
\mathscr{E}_{\epsilon,\tau,T}\ll\inf_{N_{-}\in\left(T_{n}^{-1},1\right)}\left\{ \left|\log\left(N_{-}\right)\right|^{n-1}\mathscr{N}_{\sigma,\kappa,\beta}\left(N_{-},T\right)\right\} ,\label{eq:end 4}
\end{equation}
where $\mathscr{N}_{\sigma,\kappa,\beta}\left(N_{-},T\right)=N_{-}^{\beta'd-\sigma q}+N_{-}^{-\sigma q-\left(1+\kappa\right)\beta'}T^{-\kappa'\beta'}+N_{-}^{\sigma}+N_{-}^{-\sigma q}T^{n-1-d}$.
Let $N_{-}=T^{\frac{-\kappa'\beta'}{\sigma+\sigma q+\left(1+\kappa\right)\beta'}}$
for $\sigma\geq\frac{2\beta'}{nd+1}$, we have $0<\frac{\kappa'\beta'}{\sigma+\sigma q+\left(1+\kappa\right)\beta'}\leq1$.
Thus, there exists $T_{0}>0$ such that for $T>T_{0}$ we have $N_{-}\in\left(T_{n}^{-1},1\right)$
and 
\begin{alignat}{1}
\mathscr{N}_{\sigma,\kappa,\beta}\left(N_{-},T\right) & =T^{\frac{-\kappa'\beta'\left(\beta'd-\sigma q\right)}{\sigma+\sigma q+\left(1+\kappa\right)\beta'}}+2T^{\frac{-\kappa'\beta'\sigma}{\sigma+\sigma q+\left(1+\kappa\right)\beta'}}+T^{\frac{\sigma q\kappa'\beta'}{\sigma+\sigma q+\left(1+\kappa\right)\beta'}+n-1-d}.\label{eq:end 5}
\end{alignat}
Note that for $n\leq d$, 
\begin{alignat*}{1}
\frac{\sigma q\kappa'\beta'}{\sigma+\sigma q+\left(1+\kappa\right)\beta'}+n-1-d & \leq\frac{\sigma q\kappa'\beta'}{\sigma+\sigma q+\left(1+\kappa\right)\beta'}-1=\frac{\sigma q\kappa'\beta'-\sigma-\sigma q-\left(1+\kappa\right)\beta'}{\sigma+\sigma q+\left(1+\kappa\right)\beta'}
\end{alignat*}
and since 
\[
\sigma q\kappa'\beta'-\sigma-\sigma q-\left(1+\kappa\right)\beta'+\kappa'\beta'\sigma=\sigma\left(1+q\right)\left(\kappa'\beta'-1\right)-\left(1+\kappa\right)\beta'<0,
\]
we get 
\[
\frac{\sigma q\kappa'\beta'}{\sigma+\sigma q+\left(1+\kappa\right)\beta'}+n-1-d<\frac{-\kappa'\beta'\sigma}{\sigma+\sigma q+\left(1+\kappa\right)\beta'}.
\]
Thus, from (\ref{eq:end 5}) we get 
\begin{equation}
\mathscr{N}_{\sigma,\kappa,\beta}\left(N_{-},T\right)\ll T^{\frac{-\kappa'\beta'\left(\beta'd-\sigma q\right)}{\sigma+\sigma q+\left(1+\kappa\right)\beta'}}+2T^{\frac{-\kappa'\beta'\sigma}{\sigma+\sigma q+\left(1+\kappa\right)\beta'}}.\label{eq:end 6}
\end{equation}
Let $\sigma=\frac{\beta'd}{1+q}$, then $\frac{2\beta'}{nd+1}\leq\sigma<1$,
and (\ref{eq:end 6}) becomes 
\[
\mathscr{N}_{\sigma,\kappa,\beta}\left(N_{-},T\right)\ll T^{-\frac{4d\beta'\left(1-\kappa\right)}{\left(1+nd\right)\left(d+1+\kappa\right)}}.
\]
Thus, using this and (\ref{eq:end 4}) we get that there exists $T_{0}>0$
such that for all $T>T_{0}$, 
\[
\mathscr{E}_{\epsilon,\tau}\left(T\right)\ll\log^{n-1}\left(T\right)T^{-\delta},
\]
where 
\[
\delta=\sup_{\beta'\in\left(0,1/2-2/d\right)}\frac{4d\beta'\left(1-\kappa\right)}{\left(1+nd\right)\left(d+1+\kappa\right)}>\frac{2\left(d-4\right)\left(1-\kappa\right)}{\left(1+nd\right)\left(d+1+\kappa\right)}-\varsigma
\]
for any $\varsigma>0$. Finally, we can apply Lemma \ref{lem:End},
to get that for all $T>\max\left(T_{0},\tau^{-1}\right)$, 
\[
\left|\frac{\left|\mathbb{Z}^{nd}\cap P_{Q}^{n}\left(I_{1},\dots,I_{n-1}\right)\cap B\left(T\right)\right|}{\mathrm{Vol}\left(P_{Q}^{n}\left(I_{1},\dots,I_{n-1}\right)\cap B\left(T\right)\right)}-1\right|\ll\log^{n-1}\left(T\right)T^{-\delta}.
\]
We chose $\tau=N_{-}^{\sigma}$ and then $N_{-}=T^{\frac{-\kappa'\beta'}{\sigma+\sigma q+\left(1+\kappa\right)\beta'}}$,
therefore $\tau^{-1}=T^{\frac{\sigma\kappa'\beta'}{\sigma+\sigma q+\left(1+\kappa\right)\beta'}}$.
Because $\frac{\sigma\kappa'\beta'}{\sigma+\sigma q+\left(1+\kappa\right)\beta'}<1$,
the condition that $T>\tau^{-1}$ is automatically satisfied for the
choices that were made. This completes the proof the Theorem. 
\end{proof}

\section{Volume and norm estimates\label{sec:Volume-and-norm}}

\subsection{\label{sub:Volume-estimates.}Volume estimates.}

In order to bound the smoothing errors coming from Lemmas \ref{Cor::smooth approximation 1}
and \ref{cor:smooth approx 2} we need to estimate the volumes of
certain regions of $\mathbb{R}^{nd}$. Similar computations were done
in \cite{MR2379669} (Lemma 5) for positive definite forms and in
\cite{MR1609447} (Lemma 3.8) and \cite{G-M} (Lemma 7.1) for a single
quadratic form. For $m\in\mathbb{N}$, let $S_{m}$ be the unit sphere
in $\mathbb{R}^{m}$. For $g\in GL_{d}\left(\mathbb{R}\right)$, and
$v\in\mathbb{R}^{nd}$, let $gv=\left(gv_{1},\dots,gv_{n}\right)$.
Let $f_{1}$ and $f_{2}$ be compactly supported functions on $\mathbb{R}$
and $\mathbb{R}^{n-1}$ respectively. Let 
\[
\Theta\left(f_{1},f_{2},T\right)=\int_{\mathbb{R}^{nd}}f_{2}\left(Q\left(v_{1}\right)-Q\left(v_{2}\right),\dots,Q\left(v_{n-1}\right)-Q\left(v_{n}\right)\right)f_{1}\left(T^{-1}\left\Vert v\right\Vert \right)dv.
\]
There exists $g_{0}\in GL_{d}\left(\mathbb{R}\right)$ such that $Q\left(v_{i}\right)=Q_{0}\left(g_{0}v_{i}\right)$,
where $Q_{0}$ is equal to a diagonal form with coefficients equal
to $\pm1$. Suppose that the signature of $Q_{0}$ is $\left(p,q\right)$.
Since $d\geq5$, without loss of generality we may suppose that $q\geq2$.
By making the change of variables $y_{i}=g_{0}v_{i}$, we get
\[
\Theta\left(f_{1},f_{2},T\right)=d_{Q}\int_{\mathbb{R}^{nd}}f_{2}\left(Q_{0}\left(y_{1}\right)-Q_{0}\left(y_{2}\right),\dots,Q_{0}\left(y_{n-1}\right)-Q_{0}\left(y_{n}\right)\right)f_{1}\left(T^{-1}\left\Vert g_{0}^{-1}y\right\Vert \right)dy,
\]
where $d_{Q}=1/\det\left(g_{0}\right)^{n}$. Suppose $p\geq1$ (i.e.
$Q$ is indefinite). We will work in polar coordinates. We can write
$T^{-1}y=\left(\rho_{1}\eta_{1},\dots,\rho_{2n}\eta_{2n}\right)$
where $\rho=\left(\rho_{1},\dots,\rho_{2n}\right)\in\left[0,\infty\right)^{2n}$
and $\eta=\left(\eta_{1},\dots,\eta_{2n}\right)\in\left(S_{p}\times S_{q}\right)^{n}$.
It follows that $Q_{0}\left(y_{i}\right)=T^{2}\left(\rho_{2i-1}^{2}-\rho_{2i}^{2}\right)$.
Let $f_{g}\left(v\right)=f_{1}\left(\left\Vert g^{-1}v\right\Vert \right)$
and $\overline{\rho}_{p}=\left(\prod_{i=1}^{n}\rho_{2i-1}\right)^{p-1}$,
$\overline{\rho}_{q}=\left(\prod_{i=1}^{n}\rho_{2i}\right)^{q-1}$
and $Q_{i}\left(\rho\right)=\rho_{2i-1}^{2}-\rho_{2i}^{2}-\rho_{2i+1}^{2}+\rho_{2i+2}^{2}$.
Then 
\begin{equation}
\Theta\left(f_{1},f_{2},T\right)=d_{Q}T^{nd}\int_{\left[0,\infty\right)^{2n}}\overline{\rho}_{p}\overline{\rho}_{q}f_{2}\left(T^{2}Q_{1}\left(\rho\right),\dots,T^{2}Q_{n-1}\left(\rho\right)\right)\Psi_{f_{1}}\left(\rho\right)d\rho,\label{eq:volume 1-1}
\end{equation}
where 
\begin{equation}
\Psi_{f_{1}}\left(\rho\right)=\int_{\left(S_{p}\times S_{q}\right)^{n}}f_{g_{0}}\left(\rho_{1}\eta_{1},\dots,\rho_{2n}\eta_{2n}\right)d\eta.\label{eq:volume 3-1}
\end{equation}
The following Lemma will be used to obtain the required bounds for
the smoothing errors. 
\begin{lem}
\label{lem:Volume estimates}Let $f_{1}$ be a continuous, compactly
supported function on $\mathbb{R}$ and $V$ be a bounded Borel measurable
subset of $\mathbb{R}^{n-1}$. Then, there exists a positive constant
$C_{f_{1}}$, such that 
\[
\lim_{T\rightarrow\infty}\frac{1}{T^{nd-2\left(n-1\right)}}\Theta\left(f_{1},\mathbb{1}_{V},T\right)=C_{f_{1}}\mathrm{Vol}\left(V\right).
\]
\end{lem}
\begin{proof}
Change variables in the equations (\ref{eq:volume 1-1}) and (\ref{eq:volume 3-1})
by letting $u=F\left(\rho\right)$, where $F$ is defined by 
\begin{equation}
u_{i}=\begin{cases}
\rho_{i} & \textrm{ if }i\textrm{ is odd or }i=2n\\
T^{2}\left(\rho_{i-1}^{2}-\rho_{i}^{2}\right) & \textrm{ if }i\textrm{ is even and }i\neq2n.
\end{cases}\label{eq:change of variables}
\end{equation}
Note that the Jacobian of $F$ is given by $2^{n-1}T^{2\left(n-1\right)}\prod_{i=1}^{n-1}\left|\rho_{2i}\right|$.
Therefore 
\begin{alignat*}{1}
\left|\overline{\rho}_{p}\overline{\rho}_{q}\right|d\rho & =2^{1-n}T^{2\left(1-n\right)}\left(\prod_{i=1}^{n}\left|\rho_{2i-1}\right|\right)^{p-1}\left(\prod_{i=1}^{n}\left|\rho_{2i}\right|\right)^{q-2}\left|\rho_{2n}\right|du.
\end{alignat*}
Moreover, we can write 
\[
\rho_{i}=\begin{cases}
u_{i} & \textrm{ if }i\textrm{ is odd or }i=2n\\
\sqrt{u_{i-1}^{2}-u_{i}/T^{2}} & \textrm{ if }i\textrm{ is even and }i\neq2n.
\end{cases}
\]
Therefore 
\begin{equation}
\Theta\left(f_{1},f_{2},T\right)=2^{1-n}d_{Q}T^{nd-2\left(n-1\right)}\int_{F\left(\left[0,\infty\right)^{2n}\right)}f_{2}\left(u_{2}-u_{4},\dots,u_{2n-2}-u_{2n}\right)\overline{\Psi}_{f_{1}}\left(u,T\right)du,\label{eq:volume 1-1-1}
\end{equation}
where 
\begin{equation}
\overline{\Psi}_{f_{1}}\left(u,T\right)=J\left(u,T\right)\int_{\left(S_{p}\times S_{q}\right)^{n}}f_{1}\left(\left\Vert g_{0}^{-1}\left(u_{1}\eta_{1},\sqrt{u_{1}^{2}-u_{2}/T^{2}}\eta_{2},\dots,u_{2n}\eta_{2n}\right)\right\Vert \right)d\eta\label{eq:volume 3-1-1}
\end{equation}
and 
\[
J\left(u,T\right)=\left(\prod_{i=1}^{n}\left|u_{2i-1}\right|\right)^{p-1}\left(\prod_{i=1}^{n}u_{2i-1}^{2}-u_{2i}/T^{2}\right)^{\frac{q-2}{2}}\left|u_{2n}\right|.
\]
Note that $F\left(\left[0,\infty\right)^{2n}\right)=\left(\left[0,\infty\right)\times\mathbb{R}\right)^{n-1}\times\left[0,\infty\right)\times\left[0,\infty\right)$.
Since $f_{1}$ is continuous with compact support, $f_{1}$ can be
bounded by an integrable function and hence by the Dominated Convergence
Theorem 
\[
\lim_{T\rightarrow\infty}\overline{\Psi}_{f_{1}}\left(u,T\right)=\lim_{T\rightarrow\infty}J\left(u,T\right)\int_{\left(S_{p}\times S_{q}\right)^{n}}\lim_{T\rightarrow\infty}f_{1}\left(\left\Vert g_{0}^{-1}\left(u_{1}\eta_{1},\sqrt{u_{1}^{2}-u_{2}/T^{2}}\eta_{2},\dots,u_{2n}\eta_{2n}\right)\right\Vert \right)d\eta.
\]
Since $f_{1}$ is continuous we get 
\begin{equation}
\lim_{T\rightarrow\infty}\overline{\Psi}_{f_{1}}\left(u,T\right)=\left(\prod_{i=1}^{n}\left|u_{2i-1}\right|\right)^{p+q-3}\left|u_{2n}\right|\int_{\left(S_{p}\times S_{q}\right)^{n}}f_{1}\left(\left\Vert g_{0}^{-1}\left(u_{1}\eta_{1},u_{1}\eta_{2}\dots,u_{2n}\eta_{2n}\right)\right\Vert \right)d\eta.\label{eq:vol 123}
\end{equation}
Hence, we see that $\lim_{T\rightarrow\infty}\overline{\Psi}_{f_{1}}\left(u,T\right)$
depends only on $u_{i}$ if $i$ is odd or $i=2n$. Let $\overline{\overline{\Psi}}_{f_{1}}\left(u_{1},\dots,u_{2n-1},u_{2n}\right)=\lim_{T\rightarrow\infty}\overline{\Psi}_{f_{1}}\left(u,T\right)$.
Because $f_{1}$ has compact support, it follows that the support
of $\overline{\overline{\Psi}}_{f_{1}}\left(u_{1},\dots,u_{2n-1},u_{2n}\right)$
is also compact. Let $g_{1}\in SL_{2n}\left(\mathbb{R}\right)$ be
such that $u'=g_{1}u$ where 
\[
u'_{i}=\begin{cases}
u_{i} & \textrm{ if }i\textrm{ is odd or }i=2n\\
u_{i}-u_{i+2} & \textrm{ if }i\textrm{ is even and }i\neq2n.
\end{cases}
\]
Note that $du=du'$ and $g_{1}F\left(\left[0,\infty\right)^{2n}\right)=F\left(\left[0,\infty\right)^{2n}\right)$.
Let $du'_{e}=\prod_{i=1}^{n-1}du_{2i}$, $du'_{o}=\prod_{i=1}^{n}du_{2i-1}du_{2n}$
. The fact that $f_{2}$ has compact support and thus can be bounded
by an integrable function, means that the Dominated Convergence Theorem,
together with (\ref{eq:volume 1-1-1}) and (\ref{eq:vol 123}) yields
\begin{equation}
\lim_{T\rightarrow\infty}\frac{1}{T^{nd-2\left(n-1\right)}}\Theta\left(f_{1},f_{2},T\right)=2^{1-n}d_{Q}\int_{\mathbb{R}^{n-1}}f_{2}\left(u'_{2},\dots,u'_{2n-2}\right)du'_{e}\int_{\left[0,\infty\right)^{n+1}}\overline{\overline{\Psi}}_{f_{1}}\left(u'_{1},\dots,u'_{2n-1},u'_{2n}\right)du'_{o}.\label{eq:vol 124}
\end{equation}
Therefore, by setting $f_{2}=\mathbb{1}_{V}$ we get 
\[
\lim_{T\rightarrow\infty}\frac{1}{T^{nd-2\left(n-1\right)}}\Theta\left(f_{1},\mathbb{1}_{V},T\right)=C_{f_{1}}\mathrm{Vol}\left(V\right),
\]
where 
\begin{equation}
C_{f_{1}}=2^{1-n}d_{Q}\int_{\left[0,\infty\right)^{n+1}}\overline{\overline{\Psi}}_{f_{1}}\left(u'_{1},\dots,u'_{2n-1},u'_{2n}\right)du'_{o}\label{eq:cftau}
\end{equation}
 is a positive constant, as required. \end{proof}
\begin{rem}
The case when $p=0$ (i.e. $Q$ is negative definite) can be dealt
with in the same way up to minor modifications of coordinates involved.
Specifically, we write $T^{-1}y=\left(\rho_{1}\eta_{1},\dots,\rho_{n}\eta_{n}\right)$,
where $\rho=\left(\rho_{1},\dots,\rho_{n}\right)\in\left[0,\infty\right)^{2}$
and $\eta=\left(\eta_{1},\dots,\eta_{n}\right)\in S_{d-1}^{n}$. It
follows that $Q_{0}\left(y_{i}\right)=T^{2}\rho_{i}^{2}$. The change
of variables (\ref{eq:change of variables}) also needs to be replaced
by
\[
u_{i}=\begin{cases}
\rho_{i}^{2}-\rho_{i+1}^{2} & \textrm{ if }i<n\\
\rho_{i} & \textrm{ if }i=n.
\end{cases}
\]
The Jacobian is then $2^{n-1}\prod_{i=1}^{n-1}\left|\rho_{i}\right|$,
and it is straightforward to check that the rest of the proof remains
intact for this situation. See also Lemma 5 of \cite{MR2379669}. \end{rem}
\begin{cor}
\label{cor:tau}There exists $T_{0}>0$ such that, for all $\tau\in\left(0,1/2\right)$
and $T>T_{0}$,
\[
\int_{\mathbb{R}^{nd}}\left(\mathbb{1}_{B\left(1+2\tau\right)}-\mathbb{1}_{B\left(1-2\tau\right)}\right)d\nu_{T}\ll\tau T^{nd-2\left(n-1\right)}.
\]
\end{cor}
\begin{proof}
Note that $\int_{\mathbb{R}^{nd}}\left(\mathbb{1}_{B\left(1+2\tau\right)}-\mathbb{1}_{B\left(1-2\tau\right)}\right)d\nu_{T}=\Theta\left(\mathbb{1}_{\left[1-2\tau,1+2\tau\right]},\mathbb{1}_{I_{1}\times\dots\times I_{n-1}},T\right)$.
For all $\tau>0$, there exists a continuous function $f_{\tau}$,
such that $f_{\tau}\left(x\right)=1$ for all $x\in\left[1-2\tau,1+2\tau\right]$
and $f_{\tau}\left(x\right)=0$ for all $x\notin\left[1-3\tau,1+3\tau\right]$.
We have 
\[
\Theta\left(\mathbb{1}_{\left[1-2\tau,1+2\tau\right]},\mathbb{1}_{I_{1}\times\dots\times I_{n-1}},T\right)\leq\Theta\left(f_{\tau},\mathbb{1}_{I_{1}\times\dots\times I_{n-1}},T\right).
\]
Therefore, in view of Lemma \ref{lem:Volume estimates} we must show
that $C_{f_{\tau}}\ll\tau$. Using (\ref{eq:vol 123}) and (\ref{eq:cftau})
we have 
\begin{equation}
C_{f_{\tau}}\ll\int_{\left[0,\infty\right)^{n+1}}\left(\left(\prod_{i=1}^{n}\left|u_{2i-1}\right|\right)^{p+q-3}\left|u_{2n}\right|\int_{\left(S_{p}\times S_{q}\right)^{n}}\mathbb{1}_{\left[1-3\tau,1+3\tau\right]}\left(\left\Vert g_{0}^{-1}\left(u_{1}\eta_{1},u_{1}\eta_{2}\dots,u_{2n}\eta_{2n}\right)\right\Vert \right)d\eta\right)du_{o}.\label{eq:volume 4}
\end{equation}
If $u$ is such that $\mathbb{1}_{\left[1-3\tau,1+3\tau\right]}\left(\left\Vert g_{0}^{-1}\left(u_{1}\eta_{1},u_{1}\eta_{2}\dots,u_{2n}\eta_{2n}\right)\right\Vert \right)$
is non zero then $u$ is in a bounded subset of $\mathbb{R}^{n+1}$.
Hence
\begin{equation}
C_{f_{\tau}}\ll\int_{\left[0,\infty\right)^{n+1}}\left(\int_{\left(S_{p}\times S_{q}\right)^{n}}\mathbb{1}_{\left[1-3\tau,1+3\tau\right]}\left(\left\Vert g_{0}^{-1}\left(u_{1}\eta_{1},u_{1}\eta_{2}\dots,u_{2n}\eta_{2n}\right)\right\Vert \right)d\eta\right)du_{o}.\label{eq:limit as tau}
\end{equation}
Next we change variables by letting $u_{i}=r\rho_{i}$, where $r\geq0$
and $\rho_{o}=\left(\rho_{1},\rho_{3},\dots,\rho_{2n}\right)\in S_{n+1}$.
We get
\begin{equation}
C_{f_{\tau}}\ll\int_{\left[0,\infty\right)}r^{n}\left(\int_{S_{n+1}}\left(\int_{\left(S_{p}\times S_{q}\right)^{n}}\mathbb{1}_{\left[1-3\tau,1+3\tau\right]}\left(r\left\Vert g_{0}^{-1}\left(\rho_{1}\eta_{1},\rho_{1}\eta_{2}\dots,\rho_{2n}\eta_{2n}\right)\right\Vert \right)d\eta\right)d\rho_{o}\right)dr.\label{eq:fixing the hole}
\end{equation}
 Let $N\left(\eta,\rho_{0}\right)=\left\Vert g_{0}^{-1}\left(\rho_{1}\eta_{1},\rho_{1}\eta_{2}\dots,\rho_{2n}\eta_{2n}\right)\right\Vert $.
Using Fubini's Theorem to change the order of integration in (\ref{eq:fixing the hole})
we get 
\begin{alignat*}{1}
C_{f_{\tau}} & \ll\int_{S_{n+1}}\int_{\left(S_{p}\times S_{q}\right)^{n}}\int_{\left(1-3\tau\right)/N\left(\eta,\rho_{0}\right)}^{\left(1+3\tau\right)/N\left(\eta,\rho_{0}\right)}r^{n}drd\eta d\rho_{0}\\
 & \ll\tau\int_{S_{n+1}}\int_{\left(S_{p}\times S_{q}\right)^{n}}N\left(\eta,\rho_{0}\right)^{-\left(n+1\right)}d\eta d\rho_{0}.
\end{alignat*}
By compactness $\int_{S_{n+1}}\int_{\left(S_{p}\times S_{q}\right)^{n}}N\left(\eta,\rho_{0}\right)^{-\left(n+1\right)}d\eta d\rho_{0}\ll1$,
from which the claim of the Lemma follows. \end{proof}
\begin{cor}
\label{cor:epsilon}There exists $T_{0}>0$ such that, for all $\tau\in\left(0,1\right)$,
$\epsilon\in\left(0,1\right)$ and $T>T_{0}$,
\[
\int_{\mathbb{R}^{n-1}}\left(\mathbb{1}_{I_{1}^{2\epsilon}\times\dots\times I_{n-1}^{2\epsilon}}-\mathbb{1}_{I_{1}^{-2\epsilon}\times\dots\times I_{n-1}^{-2\epsilon}}\right)d\nu_{\tau,T}\ll\epsilon T^{nd-2\left(n-1\right)}.
\]
\end{cor}
\begin{proof}
Note that for all $\tau\in\left(0,1\right)$, $w_{\tau}$ is a continuous
function with compact support on $\mathbb{R}^{nd}$, therefore there
exists a continuous function, $f$ with compact support on $\mathbb{R}$,
such that $w_{\tau}\left(v\right)\leq f\left(\left\Vert v\right\Vert \right)$
for all $v\in\mathbb{R}^{nd}$. Then $\int_{\mathbb{R}^{n-1}}\left(\mathbb{1}_{I_{1}^{2\epsilon}\times\dots\times I_{n-1}^{2\epsilon}}-\mathbb{1}_{I_{1}^{-2\epsilon}\times\dots\times I_{n-1}^{-2\epsilon}}\right)d\nu_{\tau,T}\leq\Theta\left(f,\mathbb{1}_{I_{1}^{2\epsilon}\times\dots\times I_{n-1}^{2\epsilon}}-\mathbb{1}_{I_{1}^{-2\epsilon}\times\dots\times I_{n-1}^{-2\epsilon}},T\right)$.
For all $\epsilon\in\left(0,1\right)$ we have $\textrm{Vol}\left(\textrm{supp}\left(\mathbb{1}_{I_{1}^{2\epsilon}\times\dots\times I_{n-1}^{2\epsilon}}-\mathbb{1}_{I_{1}^{-2\epsilon}\times\dots\times I_{n-1}^{-2\epsilon}}\right)\right)\ll\epsilon$,
therefore the corollary follows by applying Lemma \ref{lem:Volume estimates}. \end{proof}
\begin{cor}
\label{cor:volume}There exists a positive constant $C_{Q,n}$, depending
only on $Q$ and $n$, such that 
\[
\lim_{T\rightarrow\infty}\frac{1}{T^{nd-2\left(n-1\right)}}\mathrm{Vol}\left(P_{Q}^{n}\left(I_{1},\dots,I_{n-1}\right)\cap B\left(T\right)\right)=C_{Q,n}\prod_{i=1}^{n-1}\left|I_{i}\right|.
\]
\end{cor}
\begin{proof}
Note that $\mathrm{Vol}\left(P_{Q}^{n}\left(I_{1},\dots,I_{n-1}\right)\cap B\left(T\right)\right)=\Theta\left(\mathbb{1}_{\left[0,1\right]},\mathbb{1}_{I_{1}\times\dots\times I_{n-1}},T\right).$
The conclusion follows from Lemma \ref{lem:Volume estimates} by the
standard trick of approximating $\mathbb{1}_{\left[0,1\right]}$ from
above and below by continuous functions. 
\end{proof}

\subsection{Norm estimates\label{sub:Norms}}

In this subsection we prove estimates for $\bigl\Vert\widehat{\zeta}_{\tau}\bigr\Vert_{1}$
and the related quantity that appeared in the proof of Proposition
\ref{lem:bound on I0}. The estimates follow from standard results
about Bessel functions. Note that for any fixed $\tau>0$ the fact
that $\widehat{\zeta}_{\tau}\in\mathcal{S}\left(\mathbb{R}^{nd}\right)$
implies that $\bigl\Vert\widehat{\zeta}_{\tau}\bigr\Vert_{1}<\infty$.
In order to prove Theorem \ref{thm:Pair correlation}, this is the
only information regarding $\bigl\Vert\widehat{\zeta}_{\tau}\bigr\Vert_{1}$
that is required. However, in order to prove Theorem \ref{thm:Pair correlation-1},
an explicit bound for $\bigl\Vert\widehat{\zeta}_{\tau}\bigr\Vert_{1}$
is required. The required bound will follow from an estimate of $\left\Vert \widehat{w}_{\tau}\right\Vert _{1}$.
\begin{lem}
\label{lem:norm of w}For all $0<\tau<1$, $\left\Vert \widehat{w}_{\tau}\right\Vert _{1}\ll\tau^{\left(1-nd\right)/2}$. \end{lem}
\begin{proof}
From the definitions of $w_{\tau}$ and $k_{\tau}^{nd}$ it follows
that
\begin{equation}
\left\Vert \widehat{w}_{\tau}\right\Vert _{1}=\int_{\mathbb{R}^{nd}}\left|\hat{\mathbb{1}}_{\left[0,1\right]^{\tau}}\left(\left\Vert v\right\Vert \right)\widehat{k_{\tau}^{nd}}\left(v\right)\right|dv\leq\int_{\mathbb{R}^{nd}}\left|\hat{\mathbb{1}}_{\left[0,1\right]^{\tau}}\left(\left\Vert v\right\Vert \right)\right|\exp\left(-c\sqrt{\tau\left\Vert v\right\Vert }\right)dv.\label{eq:funny nrom 1}
\end{equation}
 Let $\mathcal{J}_{\mu}$ denote the Bessel function of order $\mu$.
See Appendix B page 425 of \cite{MR2445437}. Using the results from
B3 and B5 of \cite{MR2445437} it follows that 
\begin{equation}
\left|\hat{\mathbb{1}}_{\left[0,1\right]^{\tau}}\left(\left\Vert v\right\Vert \right)\right|=\left|\frac{\mathcal{J}_{nd/2}\left(2\pi\left(1+\tau\right)\left\Vert v\right\Vert \right)\left(1+\tau\right)^{nd/2-2}}{\left\Vert v\right\Vert ^{nd/2}}\right|\ll\left|\frac{\mathcal{J}_{nd/2}\left(2\pi\left(1+\tau\right)\left\Vert v\right\Vert \right)}{\left\Vert v\right\Vert ^{nd/2}}\right|.\label{eq:funny nrom 2}
\end{equation}
By B8 of \cite{MR2445437} we have $\mathcal{J}_{nd/2}\left(2\pi\left(1+\tau\right)\left\Vert v\right\Vert \right)\ll\left(2\pi\left(1+\tau\right)\left\Vert v\right\Vert \right)^{-1/2}\ll\left\Vert v\right\Vert ^{-1/2}$
for $\left\Vert v\right\Vert \geq\left(2\pi\left(1+\tau\right)\right)^{-1}$.
Let $r=\left(2\pi\left(1+\tau\right)\right)^{-1}$. Therefore, from
(\ref{eq:funny nrom 1}) and (\ref{eq:funny nrom 2}) we get 
\begin{alignat}{1}
\left\Vert \widehat{w}_{\tau}\right\Vert _{1} & \ll\int_{\mathbb{R}^{nd}}\left|\frac{\mathcal{J}_{nd/2}\left(2\pi\left(1+\tau\right)\left\Vert v\right\Vert \right)}{\left\Vert v\right\Vert ^{nd/2}}\right|\exp\left(-c\sqrt{\tau\left\Vert v\right\Vert }\right)dv\nonumber \\
 & \ll\int_{B\left(r\right)}\frac{\mathcal{J}_{nd/2}\left(2\pi\left(1+\tau\right)\left\Vert v\right\Vert \right)}{\left\Vert v\right\Vert ^{nd/2}}\exp\left(-c\sqrt{\tau\left\Vert v\right\Vert }\right)dv+\int_{\mathbb{R}^{nd}\setminus B\left(r\right)}\frac{\exp\left(-c\sqrt{\tau\left\Vert v\right\Vert }\right)}{\left\Vert v\right\Vert ^{\left(nd+1\right)/2}}dv.\label{eq:norm of w 1}
\end{alignat}
Moreover, by B6 of \cite{MR2445437} we get that $\mathcal{J}_{nd/2}\left(2\pi\left(1+\tau\right)\left\Vert v\right\Vert \right)/\left\Vert v\right\Vert ^{nd/2}$
is bounded for all $0<\left\Vert v\right\Vert \leq r$. Also, from
the definition in B1 of \cite{MR2445437}, it follows that when $\left\Vert v\right\Vert =0$,
this quantity is also bounded. Therefore 
\begin{equation}
\int_{B\left(r\right)}\frac{\mathcal{J}_{nd/2}\left(2\pi\left(1+\tau\right)\left\Vert v\right\Vert \right)}{\left\Vert v\right\Vert ^{nd/2}}\exp\left(-c\sqrt{\tau\left\Vert v\right\Vert }\right)dv\ll1.\label{eq:funny norm 3-1}
\end{equation}
Since $r\leq1$, by changing to polar coordinates we get 
\begin{alignat*}{1}
\int_{\mathbb{R}^{nd}\setminus B\left(r\right)}\frac{\exp\left(-c\sqrt{\tau\left\Vert v\right\Vert }\right)}{\left\Vert v\right\Vert ^{\left(nd+1\right)/2}}dv & \ll\int_{1}^{\infty}x^{\left(nd-3\right)/2}\exp\left(-c\sqrt{\tau x}\right)dx\\
 & \ll\tau^{\left(1-nd\right)/2}\int_{c\tau^{1/2}}^{\infty}x^{nd-2}\exp\left(-x\right)dx\\
 & \ll\tau^{\left(1-nd\right)/2}.
\end{alignat*}
Combining this with (\ref{eq:norm of w 1}) and (\ref{eq:funny norm 3-1})
we get the conclusion of the Lemma. \end{proof}
\begin{cor}
\label{cor:l1 norm of zeta}For all $0<\tau<1$, $\bigl\Vert\widehat{\zeta}_{\tau}\bigr\Vert_{1}\ll\tau^{\left(1-nd\right)/2}.$\end{cor}
\begin{proof}
Let $\varphi$ be a smooth function on $\mathbb{R}^{nd}$ such that
$\varphi\left(v\right)=1$ for all $v$ in the support of $w_{\tau}$
and $\varphi\left(v\right)=0$ if $\left\Vert v\right\Vert >3$. It
follows that $\zeta_{\tau}\left(v\right)=w_{\tau}\left(v\right)\exp\left(Q_{++}\left(v\right)\right)\varphi\left(v\right)$.
Let $\chi\left(v\right)=\exp\left(Q_{++}\left(v\right)\right)\varphi\left(v\right)$.
Note that because $\varphi\in C_{0}^{\infty}\left(\mathbb{R}^{nd}\right)$
and $\exp\left(Q_{++}\left(v\right)\right)$ is bounded on the support
of $\varphi$ we have that $\chi\in\mathcal{S}\left(\mathbb{R}^{nd}\right)$.
Therefore $\widehat{\chi}\in\mathcal{S}\left(\mathbb{R}^{nd}\right)$
and hence 
\begin{equation}
\left\Vert \widehat{\chi}\right\Vert _{1}<\infty.\label{eq:L1 norm 1.5-1}
\end{equation}
It is a standard fact (for example see Proposition 2.2.11 (12) of
\cite{MR2445437}) that 
\[
\int_{\mathbb{R}^{nd}}\bigl|\widehat{\zeta}_{\tau}\left(y\right)\bigr|dy\leq\int_{\mathbb{R}^{nd}}\int_{\mathbb{R}^{nd}}\left|\widehat{w}_{\tau}\left(y-v\right)\right|\left|\widehat{\chi}\left(v\right)\right|dvdy=\left\Vert \widehat{\chi}\right\Vert _{1}\left\Vert \widehat{w}_{\tau}\right\Vert _{1}.
\]
Thus, the conclusion of the Corollary follows from Lemma \ref{lem:norm of w}
and (\ref{eq:L1 norm 1.5-1}).
\end{proof}
The following Lemma was used in the proof of Proposition \ref{lem:bound on I0}.
Again we remark that knowledge of the exact dependence on $\tau$
is not necessary if one only wants to prove Theorem \ref{thm:Pair correlation}.
\begin{lem}
\label{lem:Funny norm}For all $0<\tau<1$ and $T\geq\tau^{-1}$,
\[
\int_{\mathbb{R}^{nd}\setminus B_{\infty}\left(T/2\right)}\bigl|\widehat{\zeta}_{\tau}\left(y\right)\bigr|dy\ll\frac{1}{T^{d}}\tau^{\left(1-nd\right)/2}.
\]
\end{lem}
\begin{proof}
Again, using standard results
\begin{equation}
\int_{\mathbb{R}^{nd}\setminus B_{\infty}\left(T/2\right)}\bigl|\widehat{\zeta}_{\tau}\left(y\right)\bigr|dy=\int_{\mathbb{R}^{nd}\setminus B_{\infty}\left(T/2\right)}\left(\int_{B_{\infty}\left(T/4\right)}\left|\widehat{w}_{\tau}\left(y-v\right)\right|\left|\widehat{\chi}\left(v\right)\right|dv+\int_{\mathbb{R}^{nd}\setminus B_{\infty}\left(T/4\right)}\left|\widehat{w}_{\tau}\left(y-v\right)\right|\left|\widehat{\chi}\left(v\right)\right|dv\right)dy.\label{eq:new norm 1}
\end{equation}
It is clear that 
\begin{equation}
\int_{\mathbb{R}^{nd}\setminus B_{\infty}\left(T/2\right)}\int_{\mathbb{R}^{nd}\setminus B_{\infty}\left(T/4\right)}\left|\widehat{w}_{\tau}\left(y-v\right)\right|\left|\widehat{\chi}\left(v\right)\right|dvdy\leq\left\Vert \widehat{w}_{\tau}\right\Vert _{1}\int_{\mathbb{R}^{nd}\setminus B_{\infty}\left(T/4\right)}\left|\widehat{\chi}\left(v\right)\right|dv\label{eq:new norm 2}
\end{equation}
and by changing variables $y-v=y'$, we see that if $v\in B_{\infty}\left(T/4\right)$
and $y\in B_{\infty}\left(T/2\right)$, then $y'\in\mathbb{R}^{nd}\setminus B_{\infty}\left(T/4\right)$
and hence 
\begin{equation}
\int_{\mathbb{R}^{nd}\setminus B_{\infty}\left(T/2\right)}\int_{B_{\infty}\left(T/4\right)}\left|\widehat{w}_{\tau}\left(y-v\right)\right|\left|\widehat{\chi}\left(v\right)\right|dvdy\leq\left\Vert \widehat{\chi}\right\Vert _{1}\int_{\mathbb{R}^{nd}\setminus B_{\infty}\left(T/4\right)}\left|\widehat{w}_{\tau}\left(y\right)\right|dy.\label{eq:new norm 3}
\end{equation}
Therefore, by combining (\ref{eq:new norm 1}), (\ref{eq:new norm 2})
and (\ref{eq:new norm 3}) we get 
\begin{equation}
\int_{\mathbb{R}^{nd}\setminus B_{\infty}\left(T/2\right)}\bigl|\widehat{\zeta}_{\tau}\left(y\right)\bigr|dy\leq\left\Vert \widehat{\chi}\right\Vert _{1}\int_{\mathbb{R}^{nd}\setminus B_{\infty}\left(T/4\right)}\left|\widehat{w}_{\tau}\left(y\right)\right|dy+\left\Vert \widehat{w}_{\tau}\right\Vert _{1}\int_{\mathbb{R}^{nd}\setminus B_{\infty}\left(T/4\right)}\left|\widehat{\chi}\left(v\right)\right|dv.\label{eq:new norm 4}
\end{equation}
Moreover, because $\widehat{\chi}\in\mathcal{S}\left(\mathbb{R}^{nd}\right)$,
for all $k\in\mathbb{N}$ there exists a constant $c_{k}$ such that
$\left|\widehat{\chi}\left(v\right)\right|\ll c_{k}\left(1+\left\Vert v\right\Vert ^{2}\right)^{-k}$.
By taking $k$ large enough (For instance, $k=\frac{\left(n-1\right)d}{2}-1$.)
it follows that 
\begin{equation}
\int_{\mathbb{R}^{nd}\setminus B_{\infty}\left(T/4\right)}\left|\widehat{\chi}\left(v\right)\right|dv\ll\frac{1}{T^{d}}.\label{eq:funny norm 2}
\end{equation}
We can use the method of Lemma \ref{lem:norm of w} to get 
\begin{alignat}{1}
\int_{\mathbb{R}^{nd}\setminus B_{\infty}\left(T/4\right)}\left|\widehat{w}_{\tau}\left(y\right)\right|dy & \ll\int_{\mathbb{R}^{nd}\setminus B_{\infty}\left(T/4\right)}\left|\frac{\mathcal{J}_{nd/2}\left(2\pi\left(1+\tau\right)\left\Vert v\right\Vert \right)}{\left\Vert v\right\Vert ^{nd/2}}\right|\exp\left(-c\sqrt{\tau\left\Vert v\right\Vert }\right)dv\nonumber \\
 & \ll\int_{T/4}^{\infty}r^{\left(nd-3\right)/2}\exp\left(-c\sqrt{\tau r}\right)dr\nonumber \\
 & \ll\tau^{\left(1-nd\right)/2}\int_{Tc\tau^{1/2}/4}^{\infty}x^{nd-2}\exp\left(-x\right)dx.\label{eq:funny norm 2-1}
\end{alignat}
For all $k\in\mathbb{N}$ there exists a constant $c_{k}$ such that
for $x\geq0$, $\left|x^{nd-2}\exp\left(-x\right)\right|\leq c_{k}\left(1+x\right)^{-k}$.
It follows that 
\begin{equation}
\int_{Tc\tau^{1/2}/4}^{\infty}x^{nd-2}\exp\left(-x\right)dx\ll_{k}\left(\frac{1}{\tau}\right)^{k/2}\frac{1}{T^{k}}\label{eq:new 3}
\end{equation}
for any $k\in\mathbb{N}$. Therefore, if $k>d$ and $T\geq\tau^{\frac{-k}{2k-2d}}$,
using (\ref{eq:funny norm 2-1}) and (\ref{eq:new 3}) we see 
\begin{equation}
\int_{\mathbb{R}^{nd}\setminus B_{\infty}\left(T/4\right)}\left|\widehat{w}_{\tau}\left(y\right)\right|dy\ll_{k}\frac{1}{T^{d}}.\label{eq:funny norm 3}
\end{equation}
The conclusion of the Lemma follows by choosing $k=2d$ and using
(\ref{eq:L1 norm 1.5-1}), Lemma \ref{lem:norm of w}, (\ref{eq:new norm 4}),
(\ref{eq:funny norm 2}) and (\ref{eq:funny norm 3}). 
\end{proof}
\bibliographystyle{amsalpha}
\bibliography{ref}

\end{document}